%% file: Manuscript.tex
\documentclass[10pt, twocolumn]{article}

\input{styleFile}

\usepackage{amsmath,amsthm,etoolbox}    
\usepackage{aligned-overset}
\usepackage{amssymb}    
\usepackage{comment}    
\usepackage{color}      
\usepackage[font=small,labelfont=bf]{caption}

\usepackage{hyperref}
\usepackage{xurl}
\hypersetup{breaklinks=true}
\usepackage{units}

\usepackage{graphicx}
\usepackage{xcolor} 

\usepackage{enumerate}

\graphicspath{{Figures/}}

\newtheorem{theorem}{Theorem}[section]

\newtheorem{lemma}{Lemma}[section]
\newtheorem*{definition}{Definition}%
\newtheorem*{assumption}{Assumption}

\newtheoremstyle{smallremarkstyle}
  {3pt} 
  {3pt} 
  {\small} 
  {} 
  {\itshape} 
  {.} 
  {.5em} 
  {} 

\theoremstyle{smallremarkstyle}
\newtheorem{remark}{Remark}

\newtheorem{innercustomthm}{Theorem}
\newenvironment{customthm}[1]
  {\renewcommand\theinnercustomthm{#1}\innercustomthm}
  {\endinnercustomthm}

\newcommand{\mat}[1]{\boldsymbol{#1}}
\renewcommand{\vec}[1]{\boldsymbol{#1}}

\newcommand{\mr}[1]{\mathrm{#1}}
\newcommand{\Tr}{^{\mathop{\mathrm{T}}}}

\makeatletter
\preto\env@matrix{\renewcommand{\arraystretch}{1.2}}
\makeatother

\makeatletter
\DeclareRobustCommand\vdots{%
  \mathpalette\@vdots{}%
}
\newcommand*{\@vdots}[2]{%
  \sbox0{$#1\cdotp\cdotp\cdotp\m@th$}%
  \sbox2{$#1.\m@th$}%
  \vbox{%
    \dimen@=\wd0 %
    \advance\dimen@ -3\ht2 %
    \kern.5\dimen@
    \dimen@=\wd2 %
    \advance\dimen@ -\ht2 %
    \dimen2=\wd0 %
    \advance\dimen2 -\dimen@
    \vbox to \dimen2{%
      \offinterlineskip
      \copy2 \vfill\copy2 \vfill\copy2 %
    }%
  }%
}
\DeclareRobustCommand\ddots{%
  \mathinner{%
    \mathpalette\@ddots{}%
    \mkern\thinmuskip
  }%
}
\newcommand*{\@ddots}[2]{%
  \sbox0{$#1\cdotp\cdotp\cdotp\m@th$}%
  \sbox2{$#1.\m@th$}%
  \vbox{%
    \dimen@=\wd0 %
    \advance\dimen@ -3\ht2 %
    \kern.5\dimen@
    \dimen@=\wd2 %
    \advance\dimen@ -\ht2 %
    \dimen2=\wd0 %
    \advance\dimen2 -\dimen@
    \vbox to \dimen2{%
      \offinterlineskip
      \hbox{$#1\mathpunct{.}\m@th$}%
      \vfill
      \hbox{$#1\mathpunct{\kern\wd2}\mathpunct{.}\m@th$}%
      \vfill
      \hbox{$#1\mathpunct{\kern\wd2}\mathpunct{\kern\wd2}\mathpunct{.}\m@th$}%
    }%
  }%
}
\makeatother

\title{\huge Continuation of Periodic Orbits in Conservative Hybrid Dynamical Systems and its Application to Mechanical Systems with Impulsive Dynamics}
\author{Maximilian Raff\thanks{} \and C. David Remy\samethanks}
\date{} 

\begin{document}

\allowdisplaybreaks 

\begin{titlepage}
\begin{center}
This preprint has been accepted for publication in Nonlinear Dynamics.\\[3mm]
DOI: \href{https://doi.org/10.1007/s11071-024-10565-3}{10.1007/s11071-024-10565-3}\\[5mm]
Please cite the paper as: 
Raff, M., Remy, C.D. Continuation of periodic orbits in conservative hybrid dynamical systems and its application to mechanical systems with impulsive dynamics. Nonlinear Dyn (2025).
https://doi.org/10.1007/s11071-024-10565-3\\[15mm]

First submission: 22 December 2023\\[3mm]
Major revision received: 04 June 2024\\[3mm]
Re-submission: 27 August 2024\\[3mm]
Accepted: 25 October 2024\\[3mm]
Published: 23 January 2025\\[3mm]


\end{center}
\end{titlepage}

\twocolumn[
  \begin{@twocolumnfalse}
    \maketitle
    \begin{abstract}
        In autonomous differential equations where a single first integral is present, periodic orbits are well-known to belong to one-parameter families, parameterized by the first integral's values.
        This paper shows that this characteristic extends to a broader class of conservative hybrid dynamical systems~(cHDSs).
        We study periodic orbits of a cHDS, introducing the concept of a hybrid first integral to characterize conservation in these systems. 
        Additionally, our work presents a methodology that utilizes numerical continuation methods to generate these periodic orbits, building upon the concept of normal periodic orbits introduced by Sepulchre and MacKay (1997).
        We specifically compare state-based and time-based implementations of an cHDS as an important application detail in generating periodic orbits. 
        Furthermore, we showcase the continuation process using exemplary conservative mechanical systems with impulsive dynamics.
    \end{abstract}
  \end{@twocolumnfalse}
]
\renewcommand{\thefootnote}%
    {\fnsymbol{footnote}}
  \footnotetext[1]{The authors are with the Institute for Nonlinear Mechanics, University of Stuttgart, D-70569 Stuttgart, Germany.{\tt\small \{raff, remy\}@inm.uni-stuttgart.de}\\This work was funded by the Deutsche Forschungsgemeinschaft (DFG, German Research Foundation) – 501862165. It was further supported through the International Max Planck Research School for Intelligent Systems (IMPRS-IS) for Maximilian Raff. We extend our sincere gratitude to Nelson Rosa for his contributions in past collaborations, which influenced this work.}
  \renewcommand{\thefootnote}{\arabic{footnote}}

\normalsize
\section{Introduction}
\label{sec:Intro}
Hybrid dynamical systems (HDSs) constitute an essential class of dynamical systems.
They merge continuous-time and discrete-time dynamics and are employed in modeling phenomena that exhibit continuous evolution alongside sudden discrete changes.
With these properties, they provide great utility across disciplines such as engineering, biology and economics~\cite{Schaft2000,Goebel2009,Aihara2010,Westervelt2003,Tomlin2000}.
While the applications of HDSs are wide-ranging and their dynamical behavior are generally complex, there exists a common interest in exploring specific attributes such as periodicity, autonomy, or symmetries. 
Studying these properties aims to simplify and understand the complex behaviors within HDSs~\cite{Ames2006,Ames2006a,Ames2006b,Johnson2016,Burden2015}.

This paper is focused specifically on periodicity in conservative hybrid dynamical systems.
Periodic behaviors play an important role in a wide range of applications, including, for example, astrodynamics~\cite{Montenbruck2002,Barrow-Green1997} and forced vibrations~\cite{Guckenheimer1983, Kerschen2009}.
For us, this particular emphasis is motivated by research in the domain of legged systems.
In this field, simplified mechanical models can be used to explore the fundamental principles of locomotion \cite{Garcia1998, Kuo2001, Blickhan1993, McMahon1987, McGeer1990, Remy2009}.
A subclass of these models is energetically conservative, such that the periodic orbits that underlie the gaits can be generated without any external input. 
Despite their simplicity, such conservative models showcase remarkable potential in exploring diverse families of gaits~\cite{Raff2022a,Gan2018,Geyer2006}.

Within ordinary differential equations (ODEs), conservation is a well-established concept linked to the existence of functions that remain constant along solutions, termed first integrals~\cite{Hartman2002,Arnold1992,Pontryagin1962}. 
Specifically, autonomous ODEs possessing only a single first integral are known to locally form one-dimensional (1D) families of periodic orbits~\cite{Sepulchre1997,Krauskopf2007,Albu-Schaeffer2020}. 
In the domain of HDSs, efforts have also been directed towards addressing conservation by identifying symmetries within these systems~\cite{Ames2006,Razavi2016}. These symmetries, governed by Noether's theorem, generate first integrals, thus establishing conservative dynamics~\cite{Gorni2021}. 
Beyond the general studies concerning the existence and uniqueness of solutions within HDSs~\cite{Lygeros2003,Goebel2009}, particular attention has been devoted to understanding sensitivities~\cite{Wendel2012,Bizzarri2016,Pace2017}.  
This derivative information is vital not only for local stability analyses but also for studying the local existence of periodic orbits.
While our paper does not address any general existence of periodic orbits, it is worth noting that there is ongoing but limited research in this area~\cite{Simic2002,Colombo2020}.

The objective of this work is to contribute a systematic understanding of periodic orbits in a broader scope of conservative hybrid dynamical systems (cHDSs), offering a robust framework and methodology for their analysis and generation.
The key contributions of the paper are
\begin{itemize}
    \item Introduction of a hybrid first integral to characterize cHDSs.
    \item Proof of the local existence of families of periodic orbits in cHDSs based on the construction of a Poincaré map~(Theorem~\ref{thm:continuance}).
    \item Modification of the Poincaré map to introduce a dissipation parameter, thereby regularizing its conservation-induced state dependencies.
    \item Extension of Theorem~1 in \cite{Sepulchre1997} and Lemma~6 in \cite{Munoz-Almaraz2003} to cHDS, proving that fixed points of the regularized Poincaré map are equivalent to the solution space of periodic orbits in the cHDS (Theorem~\ref{theoremXi0}).
    \item Comparative analysis between time- and state-based implementation methodologies in finding periodic orbits.
    \item Illustrative examples from mechanics, demonstrating the continuation of periodic orbits, starting from simple bifurcation points such as equilibria.
\end{itemize}

In the remainder of this paper, {we first construct recurrent trajectories for conservative hybrid dynamical system, detailing their properties.} 
This construction sets the stage for the paper's main result, Theorem~\ref{thm:continuance} (Section~\ref{sec:HDS}).
We subsequently adapt the Poincaré map to be better suited for numerical continuation, detail implementation specifics, and compare state-based and time-based numerical integration methods for constructing periodic orbits (Section~\ref{sec:Implementation}).
This continuation process is demonstrated on four distinct conservative mechanical systems with impulsive dynamics, putting a particular emphasis on the effectiveness of the time-based implementation in finding analytic initial points (Section~\ref{sec:Examples}).
Section~\ref{sec:Discus} concludes the paper, followed by an Appendix that provides the derivation of the saltation matrix (Appendix~\ref{sec:appendSaltation}), the construction of the Poincaré map's Jacobian when the the Poincaré section is placed at an event (Appendix~\ref{sec:appendixPoincare}), and detailed explanations of selected proofs (Appendix~\ref{sec:appendixProof}).
\section{Hybrid Dynamics \& Conservative Orbits}
\label{sec:HDS}
\begin{table}[!t]
\centering
\small
    \begin{tabular*}{\columnwidth}{@{\hskip 1pt} l @{\extracolsep{\fill}} l }
        \hline\\[-2mm]
         Symbol & Brief description\\[1mm]
        \hline\\[-2mm]
         $i\in\{1,\dots,m\}$ & Index of a phase (Sec. \ref{sec:hybridDyn})  \\[0.5mm]
         $m\in\mathbb{N}$ & Number of phases (Sec. \ref{sec:hybridDyn})  \\[0.5mm]
         $\Sigma =\left(\mathcal{X},\mathcal{F},\mathcal{E},\mathcal{D}\right)$ &  HDS (Sec. \ref{sec:hybridDyn})\\[0.5mm]
         $\mathcal{X}_i\subseteq \mathbb{R}^{n_i}$ & Phase domain (Sec. \ref{sec:hybridDyn})  \\[0.5mm]
         $n_i\in\mathbb{N}$ & Dimension of phase (Sec. \ref{sec:hybridDyn}) \\[0.5mm]
         $\vec f_i:\mathcal{X}_i\to T\mathcal{X}_i$ & Phase vector field (Sec. \ref{sec:hybridDyn})  \\[0.5mm]
         $\vec x_i \in \mathcal{X}_i$ & Phase state (Sec. \ref{sec:hybridDyn})\\[0.5mm]         $\mathcal{F}_i:\dot{\vec{x}}_i=\vec{f}_i(\vec{x}_i)$& Phase dynamics (Sec. \ref{sec:hybridDyn})\\[0.5mm]
         $\vec\varphi_i(t,\bar{\vec{x}}_{i})=:\vec x_i(t)$& Phase flow (Sec. \ref{sec:hybridDyn})\\[0.5mm]
         $t_i^{\text{max}}\in\mathbb{R}\cup\{\infty\}$& Maximal flow time (Sec. \ref{sec:hybridDyn})\\[0.5mm]
         $e_i^{i+1}:\mathcal{X}_i\to \mathbb{R}$ & Event function (Sec. \ref{sec:hybridDyn})  \\[0.5mm]
         $\mathcal{E}_i^{i+1}\subset \mathcal{X}_i$ & Event manifold / guard \\[0.5mm]
         $\vec{\Delta}_i^{i+1}\hspace{-1mm}:\mathcal{E}_i^{i+1}\to\mathcal{X}_{i+1}$ & Reset map (Sec. \ref{sec:hybridDyn})  \\[0.5mm]
         $\vec{x}_i^-$, $\vec{f}_i^-$ & Pre event quantitiy (Sec. \ref{sec:hybridDyn})  \\[0.5mm]
         $\vec{x}_i^+$, $\vec{f}_i^+$ & Post event quantitiy (Sec. \ref{sec:hybridDyn})  \\[0.5mm]
         $\bar{\vec x}_i\in \mathcal{X}_i,~\vec{x}_0=\bar{\vec x}_1$ & Initial conditions (Sec. \ref{sec:hybridDyn})\\[0.5mm]
         {$\mathcal{X}_0\subset \mathcal{X}_1$} & {Neighborhood of $\vec{x}_0$ (Sec. \ref{sec:hybridDyn})}\\[0.5mm]
         $t_{i}^\mathcal{E}:\mathcal{X}_i\to \mathbb{R}\cup\{\infty\}$ & {Time-to-event function} \eqref{eq:time2transition}\\[0.5mm]
         $t^\mathcal{E}$ & Accumulated event times \\[0.5mm]
         $\vec\varphi(t,\vec{x}_{0})=:\vec x(t)$ & (Recurrent) hybrid flow \eqref{eq:hybridFlow}\\[0.5mm]
         $\mathcal{T}$ & Final time interval (Sec. \ref{sec:hybridDyn})\\[0.5mm]
         $\mat{\Phi}_i\widehat{=}\nicefrac{\partial \vec{\varphi}_i}{\partial \bar{\vec{x}}_i}$
         & Fundamental matrix\\
         & of phase $i$ \eqref{eq:phasePhi}\\[0.5mm]
          $\mat{S}_i^{i+1}$ & Saltation matrix \eqref{eq:Saltation}\\[0.5mm]
          $\mat{\Phi}\widehat{=}\nicefrac{\partial \vec{\varphi}}{\partial \vec{x}_0}$ & Fundamental matrix\\
          & of the hybrid flow \eqref{eq:fundamentalMatrixHybrid}\\[0.5mm] 
          $H_i:\mathcal{X}_i\to \mathbb{R}$ & Phase first integral (Sec. \ref{sec:firstInt})\\[0.5mm]
          $k_i\in\mathbb{N}$ & Number of $H_i$ (Sec. \ref{sec:firstInt})\\[0.5mm]
          $\mathcal{H}:=\{H_i\}_1^m$ & Hybrid first integral (Sec. \ref{sec:firstInt})\\[0.5mm]
          $k_\mathcal{H}\in\mathbb{N}$ & Dimension of $\mathcal{H}$ (Sec. \ref{sec:firstInt})\\[0.5mm]
          $\Sigma^\mathcal{H}$ & Conservative HDS (Sec. \ref{sec:firstInt})\\[0.5mm]
          {$\mathcal{A}\subset \mathcal{X}_1$} & Poincaré section (Sec. \ref{sec:Poincare})\\[0.5mm]
          {$a:\mathcal{X}_1\to\mathbb{R}$} & Anchor function (Sec. \ref{sec:Poincare})\\[0.5mm]
          $t_{m+1}^\mathcal{A}:\mathcal{X}_1\to \mathbb{R}$ & {Time-to-anchor function} \eqref{eq:accumulationTimeAnchor}\\[0.5mm]
         $t^\mathcal{A}$ & Overall time to anchor \eqref{eq:accumulationTimeAnchor}\\[0.5mm]
         {$\mathcal{A}_0= \mathcal{X}_0\cap \mathcal{A}$} & Initial subset of $\mathcal{A}$ (Sec. \ref{sec:Poincare})\\[0.5mm]
         $\vec{P}:\mathcal{A}_0\to\mathcal{A}$ & Poincaré map (Sec. \ref{sec:Poincare})\\[0.5mm]
         $\mat{I}$ & Identity matrix \eqref{eq:PoincareJacobian}\\[0.5mm]
         $T\in\mathbb{R}$ & Period time \eqref{eq:periodicity}\\[0.5mm]
         $\mat{\Phi}_T:=\mat{\Phi}\left(T,\vec{x}_0\right)$ & Monodromy matrix (Sec.~\ref{sec:ConservativeOrbits})\\[0.5mm]
         $\bar{H}\in\mathbb{R}$ & Level set of $\mathcal{H}$ (Sec.~\ref{sec:ConservativeOrbits})\\[1mm]
        \hline
    \end{tabular*}
    \caption{These key symbols are used throughout this paper. Next to a short description, the symbol's first introduction is marked by a section or equation number.}
    \label{tab:keySymbols}
\end{table}
In the following, {we construct recurrent trajectories of hybrid dynamical systems and examine various of their properties.
Such recurrent trajectories start in a particular phase of the hybrid system and revisit this phase again after some time~$t$.
We examine them,} by revisiting well-established properties of ordinary differential equations~(ODEs) before extending them to the context of hybrid dynamical systems. 
These fundamental properties play a crucial role in our work and establish the groundwork for our subsequent discussion on periodic orbits.

Table \ref{tab:keySymbols} summarizes the notation and symbols used throughout this paper. 

\subsection{{Recurrent Trajectories in Hybrid Dynamical Systems}}\label{sec:hybridDyn}
{We are interested in recurrent trajectories of a class of autonomous hybrid dynamical systems~(HDSs) with~$m\in\mathbb{N}$ continuous phases and discrete transitions that switch between these phases.
By adapting the notation of \cite{Westervelt2018}, we restrict trajectories to traverse these phases in a fixed recurring order~$1\to2\to\dots\to m\to 1$, with modulo notation~$m+1=1$.}

For each phase~$i$, the system's state is denoted by~$\vec x_i \in \mathcal{X}_i$, where~$\mathcal{X}_i$ is an open connected subset of~$\mathbb{R}^{n_i}$.
The continuous evolution of the state~$\vec x_i$ in phase~$i$ is defined by the phase dynamics~$\dot{\vec{x}}_i= \vec f_i(\vec x_i)$, where the vector field~$\vec f_i:\mathcal{X}_i\to T\mathcal{X}_i$ is continuously differentiable.
The \emph{phase flow}~$\vec \varphi_i: [0,t_i^{\text{max}}) \times \mathcal{X}_i \rightarrow \mathcal{X}_i$ describes a general solution to these phase dynamics.
It characterizes the motion within phase~$i$, starting from any initial condition~$\bar{\vec{x}}_{i}\in\mathcal{X}_i$ and continuing up to a maximal time~$t_i^{\text{max}}\in\mathbb{R}\cup\{\infty\}$.
For any phase~$i$, the phase flow~$\vec \varphi_i$ is uniquely determined over a maximal interval~$[0,t_i^{\text{max}})$ due to the local Lipschitz continuity of the vector field~$\vec{f}_i$ (cf. Theorem~1.1~\cite{Hartman2002}, Theorem~3.1~\cite{Leine2004}).
For brevity, we occasionally use the notation~$\vec x_i(t):=\vec\varphi_i(t,\bar{\vec{x}}_{i})$ to {emphasize specific phase trajectories.}

Additionally, an embedded codimension-one submanifold~$\mathcal{E}_i^{i+1}$ in~$\mathcal{X}_i$, referred to as event manifold or guard, determines a transition from phase~$i$ to phase~$i+1$ via the reset map~$\vec{\Delta}_i^{i+1}:\mathcal{E}_i^{i+1}\to \mathcal{X}_{i+1}$, with~$\vec{\Delta}_i^{i+1}\in C^1$.
This phase switch occurs whenever a continuously differentiable~\emph{event~function}~$e_i^{i+1}:\mathcal{X}_i\to \mathbb{R}$ becomes zero.
Hence, the zeros of the scalar function~$e_i^{i+1}(\vec{x}_i)$ define the event manifold as
\begin{align*}
    \mathcal{E}_i^{i+1}=\left\{\vec x_i\in \mathcal{X}_i\Bigg\vert \begin{array}{l}e_i^{i+1}(\vec{x}_i)=0,\\[1mm]\dot e_i^{i+1}(\vec{x}_i)<0\end{array}\right\},
\end{align*}
where the derivative~$\mr{d}e_i^{i+1}(\vec{x}_i)$ is non-zero for all~$\vec{x}_i\in\mathcal{E}_i^{i+1}$.
Following \cite{Westervelt2018}, we define the phase-$i$ \emph{time-to-event function}~$t_{i}^\mathcal{E}:\mathcal{X}_i\to \mathbb{R}\cup\{\infty\}$, by
\begin{align}\label{eq:time2transition}    
t_{i}^\mathcal{E}(\bar{\vec{x}}_{i}):= \inf\{t\geq 0 \vert \vec \varphi_i(t,\bar{\vec{x}}_{i})\in\mathcal{E}_i^{i+1}\},
\end{align}
where, by convention, the infimum of an empty set is taken to be~$\infty$.

Adopting the notation of \cite{Grizzle2014}, we summarize the hybrid model as
{\begin{align*}%
\Sigma: \left\{ 
\arraycolsep=1.0pt
\begin{array}{llll} \mathcal{X} &= {\{\mathcal{X}_i\}}_{i=1}^m &:& \mathcal{X}_i\subset \mathbb{R}^{n_i}\\[2mm]
\mathcal{F} &= {\{\mathcal{F}_i\}}_{i=1}^m &:& \mathcal{F}_i=\left\{\dot{\vec{x}}_i= \vec f_i(\vec x_i)\right\}
\\[2mm]
\mathcal{E} &= {\{\mathcal{E}_i^{i+1}\}}_{i=1}^m &:& \mathcal{E}_i^{i+1}=\\[1mm]
  &  &  &\left\{\vec x_i\in \mathcal{X}_i\Bigg\vert \begin{array}{l}e_i^{i+1}(\vec{x}_i)=0,\\[1mm]\dot e_i^{i+1}(\vec{x}_i)<0\end{array}\right\} \\[5mm]
\mathcal{D} &= {\{\mathcal{D}_i^{i+1}\}}_{i=1}^m &:& \mathcal{D}_i^{i+1}=\\[1mm]
  &  &  &\left\{\begin{array}{l}\vec{x}_{i+1}^{+}= \vec\Delta_i^{i+1}(\vec x_{i}^-),\\[1mm]\vec x_i^-\in \mathcal{E}_i^{i+1},\\[1mm]\vec x_{i+1}^+\in \mathcal{X}_{i+1}\end{array}\right\},
\end{array}
\right..
\end{align*}}
For compactness, we represent the hybrid model as a tuple~$\Sigma =\left(\mathcal{X},\mathcal{F},\mathcal{E},\mathcal{D}\right)$.
{
\begin{definition}[Recurrent Hybrid Trajectory]
    A trajectory function~$\vec{x}: \mathcal{T} \rightarrow \mathcal{X}_1$ that begins and ends within~$\mathcal{X}_1$, while sequentially passing through all %
    phases in the fixed order~$1 \to 2 \to \dots \to m \to 1$, is termed a recurrent hybrid trajectory.
    By denoting its initial state as~$\vec{x}_0\in\mathcal{X}_1$, the trajectory of a complete cycle is recursively defined as
    \begin{subequations}\label{eq:hybridFlow}
    \begin{flalign}
        &\vec x(t) = \vec{\varphi}_1\left(t-t^{\mathcal{E}},\bar{\vec{x}}_{m+1}\right),&\label{eq:flow}\\
    &\bar{\vec{x}}_{i+1} = \vec{\Delta}_i^{i+1}\circ\vec{\varphi}_{i}\left(t_{i}^\mathcal{E}(\bar{\vec{x}}_{i}),\bar{\vec{x}}_{i}\right),~  i=1,\dots,m,&\label{eq:recursion}\raisetag{12pt}\\
    &\bar{\vec{x}}_1=\vec x_0, &
    \end{flalign}
    \end{subequations}
    where~$t\in\mathcal{T}$ is a time value for which the system is in the last phase~$m+1$ of a cycle.
    The corresponding interval is defined as~{$\mathcal{T}:=\big[t^\mathcal{E},t^\mathcal{E}+t_1^\mathcal{E}(\bar{\vec{x}}_{m+1})\big)$} with the accumulated event times~$t^{\mathcal{E}}:=\Sigma_{i=1}^m t_i^{\mathcal{E}}(\bar{\vec{x}}_{i})$.
\end{definition}
A graphical interpretation of \eqref{eq:hybridFlow} is illustrated in Fig.~\ref{fig:hybridFlow}.
In constructing a recurrent hybrid trajectory, we chose to focus on the local properties of~$\Sigma$, rather than the general existence and uniqueness of solutions as addressed in \cite{Goebel2009,Burden2016,Johnson2016}.
This local perspective is sufficient for the theoretical development in the following sections. 
However, it does require additional assumptions about the trajectory.
\begin{assumption}[Recurrent Hybrid Trajectory]
A recurrent hybrid trajectory~$\vec{x}(t)$, as defined by \eqref{eq:hybridFlow}, satisfies the following conditions:\\[-5mm]
\begin{enumerate}[{As}1]
\itemsep0.3em 
\item It is unique in forward time and yields finite time-to-event values, ensuring that $t_{i}^\mathcal{E}(\bar{\vec{x}}_{i})<t_i^{\text{max}}$ holds within each phase~$i$.
\item It transversally intersects each event manifold~$\mathcal{E}_i^{i+1}$.
\item It satisfies~$\overline{\vec\Delta_{i}^{i+1}(\mathcal{E}_{i}^{i+1})}\cap\mathcal{E}_{i+1}^{i+2}=\emptyset$ in all phases, where~$\overline{\vec\Delta_{i}^{i+1}(\cdot)}$ denotes the closure of the set~$\vec\Delta_{i}^{i+1}(\cdot)$.
\end{enumerate}
\end{assumption}
Assumption~\textit{As1} ensures a unique mapping of the reset map~$\vec{\Delta}_i^{i+1}$, complementing the previously established uniqueness of the phase flows~$\vec{\varphi}_i$ over their respective intervals~$[0,t_i^{\text{max}})$.
Additionally, the condition~$t_{i}^\mathcal{E}(\bar{\vec{x}}_{i})<t_i^{\text{max}}$ guarantees that each event manifold is reached within finite time, thereby resulting in a finite duration~$t \in \mathcal{T}$ for the recurrent hybrid trajectory~$\vec{x}(t)$.
Assumption~\textit{As2} is already implicitly satisfied by the definition of the event manifolds~$\mathcal{E}$ in the hybrid model~$\Sigma$, where 
\begin{align}\label{eq:noGrazing}
    \dot e_i^{i+1}(\vec{x}_i^-)=\mr{d} e_{i}^{i+1}(\vec x_i^-)\cdot\vec f_i(\vec x_i^-)< 0,
\end{align}
for each~$\vec x_i^-\in\mathcal{E}_{i}^{i+1}$.
Loosely speaking, the inequality~\eqref{eq:noGrazing} requires the vector field~$\vec{f}_i$ to point outward from the event manifold~$\mathcal{E}_{i}^{i+1}$.
This condition ensures that trajectories in phase~$i$ do not graze the event manifold surface~\cite{Goebel2009}. 
Similar to hypotheses HSH5 in~\cite{Westervelt2018}, Assumption \textit{As3} ensures that the result of a reset does not immediately trigger another reset, thereby preventing simultaneous multi-events.
This assumption further implies that the trajectory~$\vec{x}(t)$ traverses all phases~$i$ with a non-zero duration~$t_{i}^\mathcal{E}(\bar{\vec{x}}_{i}) > 0$.
\begin{figure}[t]
    \centering
    \includegraphics{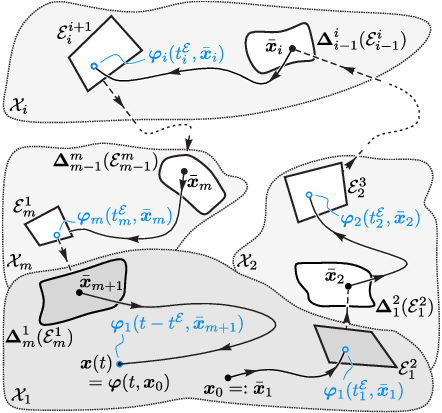}
    \caption{{}Graphical illustration of a {recurrent hybrid trajectory} of~$\Sigma$ that starts and ends in the same phase~$\mathcal{X}_1$. While all components of a HDS~$\Sigma =\left(\mathcal{X},\mathcal{F},\mathcal{E},\mathcal{D}\right)$ are recurrent, e.g., $\vec{f}_{m+1}=\vec{f}_{1}$ and~$\vec{\varphi}_{m+1}=\vec{\varphi}_{1}$, the transition states~$\bar{\vec{x}}_{i}$ are trajectory specific. For this reason, the states~$\bar{\vec{x}}_{m+1}$ and~$\bar{\vec{x}}_{1}$ are different points in general.}
    \label{fig:hybridFlow}
\end{figure}
\begin{remark}\label{remark:implicitFunctionTheorem}
    The accumulation of event times~$t^\mathcal{E}$ is well defined based on Assumptions \textit{As1}-\textit{As3}. 
    By considering \eqref{eq:noGrazing} and employing the implicit function theorem, the value of~$t^\mathcal{E}$ is determined solely by the initial state~$\vec{x}_0$ as it progresses through the hybrid trajectory~$\vec{x}(t)$. This dependency is particularly evident when analyzing the recursive structure of \eqref{eq:recursion}. Consequently, the initial state~$\bar{\vec{x}}_i$ of each phase is influenced by the overall initial state~$\vec{x}_0\in\mathcal{X}_1$. To reflect this relationship, we occasionally denote the accumulated event times as~$t^\mathcal{E}(\vec{x}_0)$ and the initial state of a phase as~$\bar{\vec x}_i(\vec{x}_0)$.
\end{remark}
{
With Assumptions~\textit{As1}-\textit{As3}, we can now locally analyze the behavior of a recurrent hybrid trajectory~$\vec{x}(t)$ of~$\Sigma$.
\begin{definition}[Recurrent Hybrid Flow]
    Given a recurrent hybrid trajectory~$\vec{x}(t)$ and under Assumptions~\textit{As1}-\textit{As3}, we locally define a recurrent hybrid flow as~$\vec \varphi: \mathcal{T} \times \mathcal{X}_0 \rightarrow \mathcal{X}_1$. 
    Specifically, within a neighborhood~$\mathcal{X}_0 \subset \mathcal{X}_1$ of the initial state~$\vec{x}_0 \in \mathcal{X}_0$, the trajectory is given by~$\vec{x}(t) := \vec{\varphi}(t, \vec{x}_0)$, where~$t \in \mathcal{T}$.
\end{definition}}
\subsection{Flow Derivatives}\label{sec:flowDiff}
In the following, we discuss the structure of phase and hybrid flow derivatives with respect to initial states $\bar{\vec{x}}_{i}$. 
In particular, we use Assumptions~\textit{As1}-\textit{As3} to indicate that the hybrid flow $\vec{\varphi}(t,\vec{x}_0)$ is continuously differentiable in both arguments (Remark~5~\cite{Grizzle2014}).

Given that~$\vec{f}_i\in C^1$, the phase flow~$\vec \varphi_i(t,\bar{\vec{x}}_{i})$ is continuously differentiable with respect to the initial state~$\bar{\vec{x}}_{i}\in\mathcal{X}_i$ (§32.6~\cite{Arnold1992}).
Consequently, the so-called fundamental matrix~$\mat{\Phi}_i$ for a phase~$i$ is well-defined by
\begin{align}
\mat{\Phi}_i\left(t,\bar{\vec{x}}_{i}\right) = \frac{\partial \vec{\varphi}_i\left(t,\vec{x}_i\right)}{\partial \vec{x}_i}\bigg\vert_{(t,\bar{\vec{x}}_{i})},~t\in\left[0,t_i^{\mathcal{E}}\right).\label{eq:phasePhi}
\end{align}
The composition property\footnote{\eqref{eq:shiftFundamental} is referred to as transition property in (Chapter~7~\cite{Leine2004}).} \cite{Dieci2011} for autonomous ODEs directly follows from a time shift $\tau$, where $0\leq\tau\leq t<t_i^\mathcal{E}$, such that 
\begin{align}\label{eq:shiftFundamental}
\mat{\Phi}_i\left(t,\bar{\vec{x}}_{i}\right) = \mat{\Phi}_i\left(t-\tau,\vec{x}_i(\tau)\right) \cdot\mat{\Phi}_i\left(\tau,\bar{\vec{x}}_{i}\right).
\end{align}
A similar chained structure appears when we compute state derivatives at time $t\in (t_i^\mathcal{E},t_i^\mathcal{E}+t_{i+1}^\mathcal{E})$ over a phase transition of $\Sigma$ (Appendix~\ref{sec:appendSaltation}):
\begin{align}\label{eq:chainedHybridDiff}
\frac{\mr{d}}{\mr{d}\vec{x}_i}\vec{\varphi}_{i+1}(t,\vec{x}_{i})\bigg\vert_{(t,\bar{\vec{x}}_{i})}
=&~\mat{\Phi}_{i+1}(t-t_i^\mathcal{E},\bar{\vec{x}}_{i+1})\notag\\
&~\cdot\mat{S}_i^{i+1}\cdot\mat{\Phi}_i(t_i^\mathcal{E},\bar{\vec{x}}_{i}),
\end{align}
where $\mat{S}_i^{i+1}$ is the so-called saltation matrix \cite{Leine2004,Burden2016,Kong2024}, which describes how perturbations of a trajectory are carried over from phase $i$ into $i+1$.
With the shorthand notation
\begin{align*}
&\vec{x}_i^-:=\vec{x}_i(t_i^\mathcal{E}(\bar{\vec{x}}_i)),\quad\vec{x}_{i+1}^+:=\bar{\vec{x}}_{i+1}\overset{\eqref{eq:recursion}}{=}\mat{\Delta}_i^{i+1}( \vec{x}_i^-),\\
&\vec{f}_i^{-}:=\vec{f}_i(\vec{x}_i^-),\quad \vec{f}_{i+1}^{+}:=\vec{f}_{i+1}(\vec{x}_{i+1}^+),\\
&\mat{D}_i^{i+1}:=\frac{\partial \mat{\Delta}_i^{i+1}(\vec{x}_i)}{\partial \vec{x}_i}\bigg\vert_{\vec{x}_i^-},~\mr{d} e_{i}^{i+1}:=\frac{\partial e_i^{i+1}(\vec{x}_i)}{\partial \vec{x}_i}\bigg\vert_{\vec{x}_i^-},
\end{align*}
the saltation matrix takes the form
\begin{align}\label{eq:Saltation}
\mat{S}_i^{i+1}=\mat{D}_{i}^{i+1} +\frac{\left(\vec{f}_{i+1}^+ - \mat{D}_{i}^{i+1}\vec{f}_i^-\right) \mr{d} e_{i}^{i+1}}
    {\mr{d} e_{i}^{i+1}\vec{f}_i^-}.
\end{align} 
The saltation matrix~\eqref{eq:Saltation} can be derived either geometrically or through calculus \cite{Mueller1995, Leine2004, Kong2024}.
In Appendix~\ref{sec:appendSaltation}, we present the latter derivation, which utilizes the chain rule and the implicit function theorem (similar to \cite{Kong2024}).
Note that Assumptions~\textit{As1}-\textit{As3} collectively ensure the well-defined nature of the saltation matrix $\mat{S}_i^{i+1}$ \cite{Bizzarri2016}. Specifically, Assumption~\textit{As2} ensures that $\dot{e}_i^{i+1}\neq 0$ in the denominator of~\eqref{eq:Saltation}.
Hence, exploiting \eqref{eq:shiftFundamental} and \eqref{eq:chainedHybridDiff} in \eqref{eq:hybridFlow}, the hybrid fundamental matrix~$\mat{\Phi}$ is well-defined for any $t\in\mathcal{T}$ and {$\vec x_0\in \mathcal{X}_0$} (Section~7.1.4~\cite{Burden2016}):
\begin{subequations}\label{eq:fundamentalMatrixHybrid}
\begin{align}
    \mat{\Phi}(t,\vec{x}_0) &=
      \frac{\partial \vec{\varphi}\left(t,\vec{x}\right)}{\partial \vec{x}}\bigg\vert_{(t,\vec{x}_0)}\\&=\mat{\Phi}_{1}(t-t^\mathcal{E},\bar{\vec{x}}_{m+1})\bar{\mat{\Phi}}_{m+1},\\
\bar{\mat{\Phi}}_{i+1} &= \mat{S}_i^{i+1} \mat{\Phi}_i(t_i^\mathcal{E},\bar{\vec{x}}_{i}),~~  i=1,\dots,m.\label{eq:recursionPhi} 
\end{align}
\end{subequations}
\begin{remark}
    While the fundamental matrix $\mat{\Phi}_i(t,\bar{\vec{x}}_{i})$ of a single phase $i$ is non-singular for all $t\in[0,t_i^{\mathcal{E}})$ (Lemma~7.1~\cite{Leine2004}), the hybrid fundamental matrix~$\mat{\Phi}(t,\vec{x}_0)$ can be singular depending on the projections of the saltation matrices $\mat{S}_{i}^{i+1}$ (Theorem~3.1~\cite{Bizzarri2016}).
    Hence, while hybrid flows~$\vec{\varphi}$ are unique in forward time~(\textit{As1}), a solution of $\Sigma$ is in general not unique in the reverse time direction. 
\end{remark}
\subsection{First Integrals and Symmetries}\label{sec:firstInt}
We further address conservation properties of these hybrid dynamical systems, starting again with a single phase.
The dynamics $\mathcal{F}_i$ of a single phase is said to be conservative if there exists a non-trivial\footnote{Trivial first integrals $H_i$ are constants with vanishing derivatives, i.e., $\mr{d}H_i\equiv\vec{0}$.}
first integral $H_i:\mathcal{X}_i\to \mathbb{R}$ of class~$C^1$, which is constant along solutions~$\vec{x}_i(t)$ \cite{Arnold1992,Pontryagin1962,Sepulchre1997}.
This yields
\begin{align}
H_i(\vec{x}_i(t))&=\text{const.} \quad \forall t\in [0,t_i^{\mathcal{E}}),\notag\\
    \overset{\tfrac{\mr{d}}{\mr{d}t}}{\Rightarrow} \dot{H}_i(\vec x_i(t))&=\mr{d}H_i(\vec x_i(t))\cdot\vec f_i(\vec x_i(t)) = 0,\label{eq:orthogonalFirstIntegral}
\end{align}
where $\mr{d} H_i(\vec x_i(t)) = \nicefrac{\partial H_i(\vec x)}{\partial \vec{x}}\vert_{\vec{x}_i(t)}$ and the initial condition $\bar{\vec{x}}_i$ is chosen arbitrarily in $\mathcal{X}_i$.
Given a conserved quantity $H_i$, $\vec{f}_i$ and $\mathcal{F}_i$ are also referred to as a conservative vector field and conservative dynamics, respectively.
\begin{remark}
    For a dynamical system and its solutions to be conservative, \eqref{eq:orthogonalFirstIntegral} must hold for all~$\vec{x}_i\in\mathcal{X}_i$.
    On the other hand, for a non-conservative system, it is still possible to find local first integrals, for which particular solutions~$\vec{x}_i(t)$ are conservative.
    In the context of periodic orbits, these solutions are typically isolated limit cycles~\cite{Christopher2007,Chatterjee2002,Gomes2005}.
\end{remark}

A first integral~$H_i$ of {conservative dynamics}~$\mathcal{F}_i$ is generally not unique, since~$\tilde{H}_i=\alpha H_i$ with~$\alpha\in\mathbb{R}\backslash\{0\}$ is also a first integral. 
Only if there exists another first integral~$\tilde{H}_i$, with~$\mr{d}\tilde{H}_i$ being pointwise linearly independent of~$\mr{d}H_i$,~$\mathcal{F}_i$ is said to have more than one (functionally) \emph{independent} first {integral}~\cite{Munoz-Almaraz2003}.
We indicate multiple~($k_i>1$) independent first integrals in phase~$i$ by an additional subscript, i.e., $\{H_{i,j}\}_{j=1}^{k_i}$.
That is, $H_{i,1}:=H_i$, $H_{i,2}:=\tilde{H}_i$, etc.
\begin{remark}\label{remark:trivialf}
     The dimension~$n_i$ of each phase domain~$\mathcal{X}_i$ must be strictly greater than~$k_i$. With~$n_i=k_i$, the first integrals~$\{H_{i,j}\}_{j=1}^{k_i}$ would completely define the phase dynamics and yield the trivial vector field~$\vec{f}_i\equiv \vec{0}$.
     This can be easily shown, since all~$H_{i,j}$ are independent and~$\mr{d}H_{i,j}\vec{f}_i=0$ holds for all~$j\in\{1,\dots,k_i=n_i\}$ at any time~$t$.
\end{remark}
\begin{definition}[Conservative
Hybrid Dynamical System]
The hybrid system~$\Sigma$ is a conservative HDS~(cHDS) with a hybrid first integral~$\mathcal{H}:=\{H_i\}_{i=1}^m$ if
\begin{enumerate}[{Df}1]
    \item each~$H_i:\mathcal{X}_i\to \mathbb{R}$ is a {continuously} differentiable non-trivial first integral of~$\mathcal{F}_i$ and
    \item it holds that~$H_{i+1}(\vec{x}_{i+1}^+) = H_{i}(\vec{x}_{i}^-)$ for all reset maps~$\vec{x}_{i+1}^+= \vec{\Delta}_i^{i+1}(\vec{x}_{i}^-)$ with~$\vec x_i^-\in \mathcal{E}_i^{i+1}$.
\end{enumerate}
\end{definition}
Note that similar to the recurrent hybrid flow, where~$\vec{f}_{m+1}=\vec{f}_{1}$, it holds~$H_{m+1}=H_{1}$.
The definition of a conservative HDS implies that for any {$\vec{x}_0\in\mathcal{X}_0$} its first integral~$H_1$ in phase~$i=1$ is invariant under the hybrid flow~$\vec{\varphi}(t,\vec{x}_0)$ for all times~$t\in\mathcal{T}$.
Furthermore, the hybrid first integral~$\mathcal{H}$ only contains first integrals~$H_i$ that also fulfill Definition~\textit{Df2}.
Loosely speaking, not all first integrals~$\{H_{i,j}\}_{j=1}^{k_i}$ of phase~$i$ are also first integrals of the conservative HDS.
{
\begin{lemma}\label{lemma:NH}
    \interlinepenalty=10000 %
    The number~$k_\mathcal{H}$ of independent first integrals~$\{H_{i,j}\}_{j=1}^{k_\mathcal{H}}$, satisfying definitions \textit{Df1} and \textit{Df2}, remains constant across hybrid phases~$i$.
    In particular, it holds~$k_\mathcal{H}\leq \{k_i\}_{i=1}^m$.
\end{lemma}}
\begin{proof}
    Under Assumption~\textit{As1}, recurrent hybrid flows are unique in forward time.
    Hence, there can not be two independent first integrals~$H_{i,1}$, $H_{i,2}$ that connect to a single first integral~$H_{i+1}$ by Definition~\textit{Df2}. That is, if phase~$i$ has~$k_\mathcal{H}$ independent first integrals that fulfill Definition~\textit{Df2}, phases~$i-1$ and $i+1$ also have~$k_\mathcal{H}$ independent first integrals.
\end{proof}
{We denote a conservative hybrid dynamical system (cHDS) by~$\Sigma^\mathcal{H}$ to emphasize the presence of a hybrid first integral~$\mathcal{H}$ within~$\Sigma$.}
\begin{remark}\label{remark:constructIntegral}
For a given dynamical system, identifying first integrals is a challenging task that remains an active area of research~\cite{Naz2014,Gorni2022}. 
Although this is beyond the scope of this paper, it is worth mentioning that one significant approach involves utilizing continuous symmetries~$\vec{\theta}:\mathcal{X}_1\to\mathcal{X}_1$~\cite{Lamb1998}.
These symmetries obey to the equation~$\vec{\theta}\circ \vec{\varphi}(t,\vec{x}_0)=\vec{\varphi}(t,\vec{\theta}\circ\vec{x}_0)$ within $\Sigma$, and their existence implies the presence of a first integral, as indicated by Noether's theorem \cite{Gorni2021}.
Of particular importance among these symmetries are so-called cyclic variables, which generate first integrals of generalized {momenta} in Hamiltonian systems \cite{Colombo2020,Ames2006a,Ames2006b}.
\end{remark}
\subsection{Transition Properties}\label{sec:transProp}
In the following, we restate known results from autonomous ODEs with first integrals \cite{Hartman2002,Arnold1992,Leine2004,Dieci2011} and show that they also hold for cHDSs $\Sigma^\mathcal{H}$.
\begin{lemma}\label{lemmaFlow}
For any conservative $\mathcal{F}_i$, with first integral $H_i:\mathcal{X}_i\to\mathbb{R}$ and its {phase trajectory}~$\vec x_i(t)$, emanating from $\bar{\vec{x}}_i$, with ${t\in[0,t_i^{\mathcal{E}})}$, it holds
\begin{subequations}
\begin{align}
    \vec f_i(\vec x_i(t))&=\mat{\Phi}_i(t,\bar{\vec{x}}_i)\vec f_i(\bar{\vec{x}}_i),\label{eq:phaseFprop}\\
    \mr{d}H_i(\bar{\vec{x}}_i)&=\mr{d}H_i({\vec x_i(t)})\mat{\Phi}_i(t,\bar{\vec{x}}_i).\label{eq:phaseHprop}
\end{align}
\end{subequations}
\end{lemma}
\begin{proof}
For a fixed initial condition~$\bar{\vec{x}}_i$, we can rewrite the autonomous flow~$\vec{\varphi}_i$ at any $t\in[0,t_i^\mathcal{E})$ using a time shift $\tau$ (§9.2~\cite{Arnold1992}) such that
\begin{align*}
\vec{x}_i(t)&=\vec{\varphi}_i(t-\tau,\vec{\varphi}_i(\tau,\bar{\vec{x}}_i)) \quad \forall \tau\in[0,t],\\
\overset{\tfrac{\mr{d}}{\mr{d}\tau}}{\Rightarrow} \vec{0} &= \dfrac{\partial \vec{\varphi}_i}{\partial t}\bigg\vert_{t-\tau}\dfrac{\partial (t-\tau)}{\partial \tau}+\dfrac{\partial \vec{\varphi}_i}{\partial \vec x}\bigg\vert_{(t-\tau,\bar{\vec{x}}_i)}\dfrac{\partial \vec{\varphi}_i}{\partial \tau}\bigg\vert_{\tau},\\
\vec{0} &=-\vec{f}_i(\vec{x}_i(t-\tau))+ \mat{\Phi}_i(t-\tau,\bar{\vec{x}}_i)\vec{f}_i(\vec{x}_i(\tau)).
\end{align*}
For $\tau=0$, this yields property~\eqref{eq:phaseFprop}.
The second property in \eqref{eq:phaseHprop} is derived by considering neighboring phase flows of a first integral $H_i$ and fixing the time $t$:
\begin{align*}
    &H_i(\vec{\varphi}_i(t,\bar{\vec{x}}_i))=H_i(\bar{\vec{x}}_i)\quad \forall \bar{\vec{x}}_i\in\mathcal{X}_i,\\
    &\overset{\tfrac{\mr{d}}{\mr{d}\bar{\vec{x}}_i}}{\Rightarrow} \frac{\partial H_i(\vec{x}_i)}{\partial \vec{x}_i}\bigg\vert_{\vec x_i(t)}\frac{\partial \vec{\varphi}_i(t,\vec{x}_i)}{\partial \vec{x}_i}\bigg\vert_{(t,\bar{\vec{x}}_i)} = \dfrac{\partial H_i(\vec{x}_i)}{\partial \vec x_i}\bigg\vert_{\bar{\vec{x}}_i}. 
\end{align*}
\end{proof}
{
\begin{lemma}\label{theoremHybrid}
For a cHDS $\Sigma^\mathcal{H}$, with first integrals $H_i\in\mathcal{H}$, and its recurrent hybrid trajectory~$\vec x(t)$, with {$t\in\mathcal{T}$}, it holds
\begin{subequations}
\begin{align}
    \vec f_{1}(\vec{x}(t))&=\mat{\Phi}(t,\vec x_0)\vec f_1(\vec x_0),\label{eq:hybridFprop}\\
    \mr{d}H_1(\vec x_0)&=\mr{d}H_1(\vec{x}(t))\mat{\Phi}(t,\vec{x}_0).\label{eq:hybridHprop}
\end{align}
\end{subequations}
\end{lemma}}
\begin{proof}
Because of the chained structure of $\mat{\Phi}$ in \eqref{eq:recursionPhi} and the properties from Lemma \ref{lemmaFlow}, it remains to be shown that
\begin{align*}
\vec{f}_{i+1}^+=\mat{S}_i^{i+1}\vec{f}_i^-,\quad \mr{d}H_i(\vec x_i^-)= \mr{d}H_{i+1}(\vec{x}_{i+1}^+)\mat{S}_i^{i+1},
\end{align*}
holds for all $i\in\{1,\dots,m\}$.
The first condition is readily confirmed by simplifying the terms and fractions in \eqref{eq:Saltation}. Similarly, employing \eqref{eq:Saltation} and \eqref{eq:orthogonalFirstIntegral} allows the second condition to be simplified as $\mr{d}H_i(\vec x_i^-)=\mr{d}H_{i+1}(\vec x_{i+1}^+)\mat{D}_i^{i+1}$.
This is demonstrated by computing the total derivative of Definition~\textit{Df2}:
\begin{align*}
    &H_{i+1}\circ\vec{\Delta}_i^{i+1}(\vec{x}_{i}^-) = H_{i}(\vec{x}_{i}^-)\quad \forall \vec{x}_{i}^-\in\mathcal{E}_i^{i+1},\\
    &\overset{\tfrac{\mr{d}}{\mr{d}\vec{x}_i^-}}{\Rightarrow} \mr{d}H_{i+1}(\vec x_{i+1}^+)\mat{D}_i^{i+1} =  \mr{d}H_i(\vec x_i^-). 
\end{align*}
\end{proof}
\subsection{First Recurrence Map}\label{sec:Poincare}
Note that \eqref{eq:phaseFprop} and \eqref{eq:hybridFprop} can be considered as transitions in time since a state remains always on the same flow when there is solely a change in time.
We are, however, interested in distinct recurrent hybrid flows which are obtained from local changes that are transversal to the flow.
For this purpose, we introduce the well-known \emph{Poincaré map} (Chapter~9.1~\cite{Leine2004}) for recurrent hybrid flows.
{
\begin{definition}[Poincaré section]
   Let a continuously differentiable function~$a:\mathcal{X}_1\to \mathbb{R}$ define the set~$\mathcal{A}=\{\vec{x}_1\in\mathcal{X}_1\vert a(\vec{x}_1)=0,\dot{a}(\vec{x}_1)<0\}$. Within a recurrent hybrid trajectory~$\vec{x}(t)$ of $\Sigma$, the Poincaré section~$\mathcal{A}$ is locally defined at points where $\vec{x}(t)\in\mathcal{A}$.
\end{definition}
Similar to \cite{Raff2022a}, we refer to the implicit function~$a$ as the \emph{anchor}.}
An illustration of the Poincaré section~$\mathcal{A}$ is depicted in Fig.~\ref{fig:PonincareSection}.
\begin{figure}[t]
    \centering
    \includegraphics{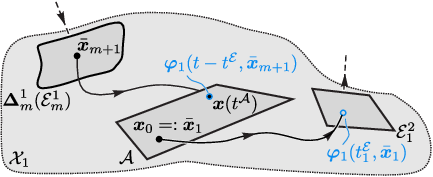}
    \caption{{}Graphical illustration of a recurrent hybrid trajectory in phase $\mathcal{X}_1$ that starts and ends on the Poincaré section~$\mathcal{A}$.}
    \label{fig:PonincareSection}
\end{figure}

\begin{remark}\label{remark:eventanchor}
    Without loss of generality, we here assume that $\mathcal{A}\cap\mathcal{E}_m^1=\emptyset$.
    That is, the \emph{anchor} is not identical to an \emph{event} of the HDS.
    This is different from the definitions in \cite{Grizzle2014,Westervelt2018}, which explicitly use $a := e_m^{1}$.
    How this modifies the derivative of the hybrid flow is discussed in detail in Appendix~\ref{sec:appendixPoincare}.
\end{remark}
With $\vec \varphi(t,\vec{x}_{0})\in\mathcal{A}$, the time~$t\in\mathcal{T}$ is coupled to the initial condition $\vec{x}_0$.
This is easily seen by constructing a \emph{time-to-anchor} function~$t_{m+1}^\mathcal{A}(\bar{\vec{x}}_{m+1}):=\inf\{t\geq 0 \vert \vec\varphi_{1}(t,\bar{\vec{x}}_{m+1})\in\mathcal{A}\}$, which is similar to
the time-to-event functions~$t_i^\mathcal{E}(\bar{\vec{x}}_i)$.
Because of the recursive structure of the hybrid trajectory~\eqref{eq:recursion} and Remark~\ref{remark:implicitFunctionTheorem},
we obtain the overall accumulated time as
\begin{align}\label{eq:accumulationTimeAnchor}
t^\mathcal{A}(\vec{x}_0) :&= t_{m+1}^\mathcal{A}(\bar{\vec{x}}_{m+1})+t^\mathcal{E}(\vec{x}_0)\\
&= t_{m+1}^\mathcal{A}\big(\bar{\vec{x}}_{m+1}(\vec{x}_0)\big)+\sum_{i=1}^m t_i^{\mathcal{E}}\big(\bar{\vec{x}}_{i}(\vec{x}_0)\big).\notag
\end{align}
Furthermore, note that $t^\mathcal{A}(\vec{x}_0)$ is continuously differentiable due to the $C^1$ property of~$\Sigma$ and of the anchor~$a$. 
In the following discussion, and without loss of generality, we consider Poincaré sections~$\mathcal{A}$ for which the accumulated time~$t^{\mathcal{A}}(\vec{x}_0)$ is finite for any point $\vec{x}_0 \in \mathcal{A}_0$, where $\mathcal{A}_0= \mathcal{A}\cap\mathcal{X}_0$.
\begin{definition}[Poincaré map]   
A state~$\vec{x}_0$ from the subset $\mathcal{A}_0=\mathcal{A}\cap\mathcal{X}_0$ is mapped back onto the Poincaré section~$\mathcal{A}$ along the recurrent hybrid trajectory by the Poincaré~map $\vec{P}:\mathcal{A}_0\to\mathcal{A}$, where $\vec{P}(\vec{x}_0) =\vec{\varphi}\big(t^\mathcal{A}(\vec{x}_0),\vec{x}_0\big)$ is also called the \textit{first recurrence map}. 
\end{definition}
Note that, similar to the recurrent hybrid trajectory~$\vec{x}(t)$, the Poincaré map~$\vec{P}\in C^1$ is well defined on the set~$\mathcal{A}_0$.
However, as the time~$t$ was eliminated by the introduction of the Poincaré section~$\mathcal{A}$, there can be no more transitions in time~\eqref{eq:hybridFprop} for the Poincaré map's Jacobian~$\mr{d}\vec{P}(\vec{x}_0)$.
Instead, the Jacobian takes the form
{\begin{align}\label{eq:PoincareJacobian}
\mr{d}\vec{P}(\vec{x}_0)&=\frac{\mr{d}}{\mr{d} \vec{x}_0}\vec{\varphi}\big(t^\mathcal{A}(\vec{x}_0),\vec{x}_0\big)\\
&=\vec{f}_1\left(\vec{x}(t^\mathcal{A})\right)\dfrac{\partial t^\mathcal{A}}{\partial \vec{x}_0}+\mat{\Phi}\left(t^\mathcal{A},\vec{x}_0\right)\notag\\
&= \left(\mat{I}-\frac{\vec{f}_1\left(\vec{x}(t^\mathcal{A})\right)\mr{d} a}{\mr{d} a \cdot\vec{f}_1\left(\vec{x}(t^\mathcal{A})\right) }\right)\mat{\Phi}\left(t^\mathcal{A},\vec{x}_0\right),\notag
\end{align}}     
where $\mr{d} a = \nicefrac{\partial a(\vec x)}{\partial \vec{x}}\vert_{\vec{x}(t^\mathcal{A})}$ and we used the implicit function theorem (see also appendix \eqref{eq:nablaT}) to derive
\begin{align*}
\dfrac{\partial t^\mathcal{A}}{\partial \vec{x}_0}=-\frac{\mr{d} a}{\mr{d} a \cdot\vec{f}_1\left(\vec{x}(t^\mathcal{A})\right) }\mat{\Phi}\left(t^\mathcal{A},\vec{x}_0\right).
\end{align*}
The absence of time-dependent transitions in the Jacobian~$\mr{d}\vec{P}(\vec{x}_0)$ is an expected outcome by introducing the Poincaré section, since the vector field~$\vec{f}_1$ points outside of~$\mathcal{A}$.
Consequently, in the local search for neighboring trajectories~$\vec{x}(t)$, perturbations along the vector field~$\vec{f}_1$ are no longer admissible when using a Poincaré section~$\mathcal{A}$.
\subsection{Fixed Points \& Periodic Orbits}
\label{sec:ConservativeOrbits}
In this section, our focus centers on the examination of fixed points denoted as $\vec{x}_0\in \mathcal{A}_0$ of the Poincaré map $\vec{P}$, i.e., $\vec{x}_0=\vec{P}(\vec{x}_0)$.
The main theoretical contribution of this paper is to show that such a fixed point is not isolated for cHDSs (Theorem~\ref{thm:continuance}). 
To achieve this, we explore admissible perturbations within $\mathcal{A}$, necessitating the utilization of the subsequent Lemma.
\begin{lemma}\label{lemma:nonzerodh}
    Given a fixed point $\vec{x}_0$ of $\vec{P}$ (under Assumptions \textit{As1}-\textit{As3}), it holds that
    \begin{itemize}
        \item the phase vector field $\vec{f}_i$, and
        \item the derivative $\mr{d}H_i$, with $H_i\in\mathcal{H}$,
    \end{itemize}
     are nonzero in all phases $i$.
\end{lemma}
\begin{proof}
Due to the finite execution time~$t_{i}^\mathcal{E}(\bar{\vec{x}}_{i})$ of phase trajectories~$\vec{x}_i(t)$ (\textit{As1}), which are unique in forward and backward time, as well as the transversality condition~(\textit{As2}), it is ensured that $\vec{f}_i\neq\vec{0}$ holds for all phases $i\in\{1,\dots,m\}$. 
Furthermore, applying the same line of reasoning, we conclude that $\vec{f}_{m+1}$ is also nonzero due to the equality of vector fields $\vec{f}_{m+1}(\vec{x}(t^\mathcal{A}))=\vec{f}_{1}(\vec{x}_0)$ with the fixed point~$\vec{x}_0=\vec{x}(t^\mathcal{A})$.
The property $\mr{d}H_i\neq \vec{0}$ directly follows from $\vec{f}_i\neq\vec{0}$ and Theorem~1.1 in~\cite{Kozlov2020}.
\end{proof}
\begin{theorem}\label{thm:continuance}
Given a cHDS~$\Sigma^\mathcal{H}$, a fixed point of the Poincaré map~$\vec{P}$ is never isolated. 
\end{theorem}
\begin{proof}
For a fixed point $\vec{x}_0\in\mathcal{A}$ that solves $\vec{P}(\vec{x}_0)-\vec{x}_0=\vec{0}$, we assume that it is isolated.
In other words, the kernel and cokernel of the square Jacobian $\mr{d}\vec{P}(\vec{x}_0)-\vec{I}$ is assumed to be trivial and thus, 
\begin{align*}
  \text{dim}(\text{ker}(\mr{d}\vec{P}(\vec{x}_0)-\vec{I}))=\text{dim}(\text{ker}(\mr{d}\vec{P}(\vec{x}_0)\Tr-\vec{I}))=0.  
\end{align*}
However, with the properties \eqref{eq:orthogonalFirstIntegral} and \eqref{eq:hybridHprop}, the row vector $\mr{d}H_1(\vec{x}(t^\mathcal{A}))$ appears to be a left eigenvector of the Poincaré map's Jacobian \eqref{eq:PoincareJacobian}, such that
\begin{flalign}\label{eq:leftEigen}
&\mr{d}H_1(\vec{x}(t^\mathcal{A}))\mr{d}\vec{P}(\vec{x}_0)=\mr{d}H_1(\vec{x}_0)=\mr{d}H_1(\vec{x}(t^\mathcal{A})),&\raisetag{12pt}
\end{flalign}
where we exploited periodicity $\vec{x}(t^\mathcal{A})=\vec{x}_0$ in the last equality of \eqref{eq:leftEigen}.
With \eqref{eq:leftEigen} it is now easy to deduce that $\mr{d}H_1(\vec{x}(t^\mathcal{A}))\Tr\neq \vec{0}$ (Lemma~\ref{lemma:nonzerodh}) is in the kernel of~$\mr{d}\vec{P}(\vec{x}_0)\Tr-\vec{I}$.
This is a contradiction to the assumption of a trivial kernel.
\end{proof}
Note that every fixed point of the Poincaré map is directly related to a periodic orbit with period time $T=t^\mathcal{A}$, such that 
\begin{align}\label{eq:periodicity}
    \vec{\varphi}(T,\vec{x}_0) = \vec{x}_0,
\end{align}
with its local linearization $\mat{\Phi}_T:=\mat{\Phi}\left(T,\vec{x}_0\right)$\footnote{The monodromy matrix~$\mat{\Phi}_T$ is an important tool to study the stability and local existence of periodic orbits (Chapter~7.1~\cite{Leine2004}).
Herein, the eigenvalues of $\mat{\Phi}_T$ are the so called Floquet multipliers.}, called the monodromy matrix.
With the well known \emph{freedom of phase} property~$\mat{\Phi}_T \vec{f}_1(\vec{x}_0) = \vec{f}_1(\vec{x}_0)$ of autonomous systems, due to \eqref{eq:hybridFprop} and \eqref{eq:periodicity}, $\mat{\Phi}_T$ has at least two eigenvalues of 1. 
Hence, Theorem~\ref{thm:continuance} also implies that periodic orbits are {non-isolated} in a cHDS.

\begin{remark}
    {By using an anchor~$a$ and introducing a Poincaré section~$\mathcal{A}$, Theorem~\ref{thm:continuance} offers a crucial step towards the numerical exploration of periodic orbits in HDSs. 
    This method of proving the existence of a family of orbits differs from the original proof~(Theorem~4~\cite{Sepulchre1997}) for ODEs, which relies solely on the properties of the periodic trajectory~\eqref{eq:periodicity} and the monodromy matrix~$\mat{\Phi}_T$.}
    However, events are already an essential part of HDSs, rendering the establishment of a Poincaré section a straightforward procedure~(see also Remark~\ref{remark:eventanchor}).
    Furthermore, this construction is crucial for the numerical exploration of periodic gaits, as discussed in Section~\ref{sec:ContinuationOrbits}.
\end{remark}
Because of the rank deficiency of the Jacobian~$\mr{d}\vec{P}(\vec{x}_0)-\vec{I}$, a periodic orbit of a conservative system is called degenerate (Definition~1~\cite{Sepulchre1997}).
For the targeted exploration of neighboring periodic orbits, we are particularly interested in the weakest form of degeneration.
That is, there are distinct neighboring periodic orbits that lie on adjacent level sets~$\bar{H}_i$ of $H_i\in\mathcal{H}$.
\begin{definition}[Normal Conservative Orbit]
A hybrid periodic trajectory of a cHDS $\Sigma^\mathcal{H}$ is said to be \emph{normal} if 1 is a simple eigenvalue of the Poincaré map's Jacobian $\mr{d}\vec{P}(\vec{x}_0)$ and thus, $\mr{dim}(\mr{ker}(\mr{d}\vec{P}(\vec{x}_0)-\vec{I}))=1$.
\end{definition}
The definition is analog to Definition~3 in \cite{Sepulchre1997} but generalizes to hybrid dynamical systems.
It implies that a normal conservative orbit yields a monodromy matrix $\mat{\Phi}_T$ with exactly two eigenvalues of 1.
Furthermore, there only exists a single~($k_1=1$) hybrid first integral~$H_1\in\mathcal{H}$ in~$\mathcal{X}_1$.
From Lemma~\ref{lemma:NH}, we deduce that all~$H_i\in\mathcal{H}$ are single hybrid first integrals and thus, $k_\mathcal{H}=1$.
With Theorem \ref{thm:continuance}, it is easy to infer that a point~$\vec{x}_0\in\mathcal{A}$ corresponding to a normal conservative orbit lies on a locally defined one-dimensional submanifold (curve) parameterized by the level set~$\bar{H}$.
From this follows that any point on the curve is uniquely defined by the implicit equations 
\begin{align}\label{eq:implicitPoincare}
    \begin{bmatrix}
        \vec{P}(\vec{x}_0)-\vec{x}_0\\
        H_1(\vec{x}_0)-\bar{H}
    \end{bmatrix}=\vec{0}.
\end{align}
A numerical construction of the curve defined by \eqref{eq:implicitPoincare} as well as the detection of turning and bifurcation points for which a periodic orbit ceases to be normal, is described in the next Section \ref{sec:Implementation}. 

\begin{remark}
    Theorem~\ref{thm:continuance} focuses exclusively on the local existence of additional periodic orbits, assuming the presence of periodic solutions. 
    In this paper, we don't address the general existence of periodic orbits as this strongly depends on specific properties of~$\Sigma^\mathcal{H}$.
    However, there are some important sub-classes of~$\Sigma^\mathcal{H}$ that are worth mentioning.
    In Remark~\ref{remark:trivialf}, we have already indicated that for $k_\mathcal{H}=n_1$ the only solutions for $\Sigma^\mathcal{H}$ are equilibria, which however, violate our Assumptions \textit{As1}-\textit{As3}.
    Moreover, the existence of symmetric periodic solutions has already been established in cHDS with $m=1$\footnote{Single phase models $\Sigma$ with $m=1$ are often referred to as simple hybrid systems~\cite{Ames2006}.} and a so-called time-reversal symmetry $\vec{\kappa}:\mathcal{X}_1\to\mathcal{X}_1$, satisfying~$\vec{\kappa}\circ \vec{\varphi}(t,\vec{x}_0)=\vec{\varphi}(-t,\vec{\kappa}\circ\vec{x}_0)$, $\vec{\kappa}\circ\vec{\kappa}=\text{id.}$ and~$\vec{\Delta}_1^2=\vec{\kappa}$~\cite{Colombo2020,Hyon2005,Lamb1998}.\footnote{The concept of time-reversal symmetry is strongly connected to continuous symmetries (Remark~\ref{remark:constructIntegral}). For more information, see~\cite{Lamb1998}.}
\end{remark}

\section{Numerical Exploration of Conservative Orbits}
\label{sec:Implementation}
The connectedness of (normal) conservative orbits, as it was shown in theory in the previous section, can be harnessed in practice through numerical continuation methods~\cite{Allgower2003}.
To this end, we adapt and extend the formulation introduced previously to make it better suited for numerical methods.
Specifically, we expand concepts from~\cite{Sepulchre1997} to incorporate hybrid dynamics and embed a cHDS within a one-parameter family of dissipative dynamics.
Additionally, we provide implementation details for an effective exploration of conservative orbits. 
Among other adjustments, we relax some of the previously stated requirements to achieve a more efficient implementation.

\subsection{Predictor-Corrector Methods}
\label{sec:PCmethods}
Let us first establish some basic concepts of continuation methods, in particular of so called predictor-corrector methods.
Here, we consider maps of the form $\vec{r}:\mathbb{R}^{N+1}\to\mathbb{R}^N$, which are assumed to be sufficiently differentiable ($\vec{r}\in C^p$) and implicitly define surfaces of the form:
\begin{align}
    \mathcal{S}:= \vec{r}^{-1}(\vec{0})=\left\{\vec{u}\in\mathbb{R}^{N+1}~\vert~\vec{r}(\vec{u})=\vec{0}\right\}.
\end{align}
We assume there exists a point $\vec{u}^\ast\in\mathcal{S}$ for which the Jacobian $\mr{d}\vec{r}(\vec{u}^\ast)$ has maximum rank, i.e., $\text{rank}(\mr{d}\vec{r}(\vec{u}^\ast))=N$.
Such a point $\vec{u}^\ast$ is called a \emph{regular} point.
By the implicit function theorem, there exists a $p$-times differentiable curve $\vec{c}(s) \in \mathcal{S}$, with $\vec{c}(0)=\vec{u}^\ast$, in the open interval $s\in (-\varepsilon,\varepsilon)$ for some {$\varepsilon>0$ \cite{Allgower2003}}.
Note that for a regular point $\vec{u}^\ast$, the tangent space $T_{\vec{u}^\ast}\mathcal{S}$ is equivalent to the kernel $\text{ker}(\mr{d}\vec{r}(\vec{u}^\ast))$, since it holds for all $s\in (-\varepsilon,\varepsilon)$:
\begin{align}
    \vec{r}(\vec{c}(s) )=\vec{0}\quad \overset{\tfrac{\mr{d}}{\mr{d}s}}{\Rightarrow} \quad\mr{d}\vec{r}(\vec{c}(s))\cdot\dot{\vec{c}}(s)=\vec{0},
\end{align}
and in particular, $\mr{d}\vec{r}(\vec{u}^\ast)\dot{\vec{u}}^\ast=\vec{0}$, with $\dot{\vec{u}}^\ast=\dot{\vec{c}}(0)$.
To explore this curve numerically, the curve is parameterized with respect to the arc-length parameter $s$ such that $\lVert\dot{\vec{c}}(s)\rVert_2=1$.
Similar to Definition~(2.1.7) in \cite{Allgower2003}, the tangent vector $\vec{\tau}\circ\mr{d}\vec{r}(\vec{c}(s))$ is said to be  \emph{induced by $\mr{d}\vec{r}(\vec{c}(s))$} with
\begin{flalign}\label{eq:tangentvector}
    &\vec{\tau}:=\left\{ \vec{\tau}\in\mathbb{R}^{N+1}\left\vert\begin{array}{cc}
        \mr{d}\vec{r}(\vec{c}(s))\cdot\vec{\tau} & =\vec{0}, \\[1mm]
         \lVert\vec{\tau}\rVert_2 & = 1, \\[1mm]
         \text{det}\left(\begin{bmatrix}
             \mr{d}\vec{r}(\vec{c}(s))\\\vec{\tau}\Tr
         \end{bmatrix}\right) & >0.
    \end{array}\right.\right.&\raisetag{36pt}
\end{flalign}
Note that with the third condition in \eqref{eq:tangentvector}, the vector $\vec{\tau}$ is uniquely defined for any regular point~$\vec{c}(s)$.
This condition on the determinant is needed to consistently define the orientation of the curve~$\vec{c}(s)$.
Since it is possible to trace a curve in two directions, we add a scalar $d\in\{-1,1\}$ to the resulting initial value problem:
\begin{align}\label{eq:IVP}
\dot{\vec{u}}=d\cdot\vec{\tau}\circ\mr{d}\vec{r}(\vec{u}),\quad \vec{u}(0)=\vec{u}_0.
\end{align}
\begin{figure}[t]
    \centering
    \includegraphics{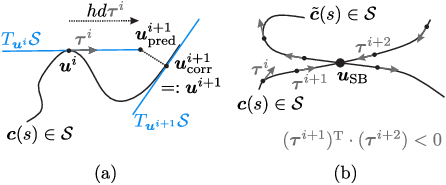}
    \caption{(a) illustrates a predictor-corrector step as described by \eqref{eq:predictor} and \eqref{eq:corrector}. (b) illustrates how a predictor-corrector step jumps over a simple bifurcation (SB) point $\vec{u}_\text{SB}$ and the tangent vectors $\vec{\tau}$ flip orientations.
    Both illustrations abbreviate tangent vectors by $\vec{\tau}^{i}:=\vec{\tau}\circ\mr{d}\vec{r}(\vec{u}^i)$.}
    \label{fig:pcMethod}
\end{figure}
In a predictor-corrector method, there is an initial predictor step that typically approximates the ODE in \eqref{eq:IVP} with a forward Euler and step size~$h>0$ such that
\begin{align}\label{eq:predictor}
    \vec{u}_\text{pred}^{i+1} = \vec{u}^{i} +hd\vec{\tau}\circ\mr{d}\vec{r}(\vec{u}^{i}).
\end{align}
The step size $h$ can be chosen adaptively to trade-off accuracy and speed in tracing $\vec c(s)$ (Chapter~6~\cite{Allgower2003}).
The subsequent corrector step utilizes local contraction properties of Newton's method to solve $\vec{r}(\vec{u})=\vec{0}$.
Hence, the corrector step solves approximately the optimization problem:
\begin{align}\label{eq:corrector}
    \vec{u}^{i+1}:=\vec{u}_\text{corr}^{i+1} \approx \underset{\vec{u}\in\mathcal{S}}{\text{argmin}}\left\lVert\vec{u}-\vec{u}_\text{pred}^{i+1}\right\rVert_2.
\end{align}
The predictor \eqref{eq:predictor} and corrector \eqref{eq:corrector} steps are illustrated in Fig.~\ref{fig:pcMethod}a.
Various alternatives to \eqref{eq:predictor} and \eqref{eq:corrector} are outlined in (Chapter~2.27~\cite{Keller1987}).
These steps are executed repeatedly until the iterative process of tracing the curve $\vec c(s)$ is terminated by the user or by the crossing of so called \emph{simple bifurcation} (SB) points.
So far, we assumed for \eqref{eq:IVP} that all points $\vec{c}(s)$ are regular.
However, when tracing a curve $\vec{c}(s)$ in $\mathcal{S}$, the predictor-corrector method occasionally crosses over simple bifurcation points $\vec{u}_\mr{SB}$ for which $\text{rank}(\mr{d}\vec{r}(\vec{u}_\mr{SB}))=N-1$ and thus, $\text{det}\left(\left[\begin{smallmatrix}
             \mr{d}\vec{r}(\vec{u}_\mr{SB})\\\dot{\vec{u}}_\mr{SB}\Tr
         \end{smallmatrix}\right]\right) =0$.
Hence, the point~$\vec{u}_\mr{SB}$ is singular.
As illustrated in Fig.~\ref{fig:pcMethod}b, a useful indicator for SB points are the comparison of adjacent tangent vectors, since they flip their orientation at an SB point (Chapter~8.1~\cite{Allgower2003}).
Continuing the tracing of the curve $\vec{c}(s)\in\mathcal{S}$ past an SB point can be achieved by a simple redefinition of the direction $d:=-d$ in the predictor~\eqref{eq:predictor}. 
However, transitioning to a new curve~$\tilde{\vec{c}}(s)\in\mathcal{S}$ necessitates additional consideration.
For the identification of tangent vectors at a bifurcating point~$\vec{u}_\mr{SB}$, we employ the Lyapunov-Schmidt reduction (Chapter~8.3~\cite{Allgower2003}).
It is crucial to emphasize that a singular point does not inherently indicate a bifurcation point, meaning the emergence of additional linearly independent tangent vectors. 
Therefore, special attention is required to properly construct tangent vectors at a singular point (Chapter~5~\cite{Keller1987}).

\begin{remark}
    Aside from the presented predictor-corrector methods, there is another class of piecewise-linear continuation methods that do not rely on the differenitability of~$\vec{r}$.
    However, since we have a differentiable Poincaré map~$\vec{P}$ within \eqref{eq:implicitPoincare}, predictor-corrector methods are generally more favorable due to their numerical efficiency.
    For further reading about general numerical continuation methods, we refer to \cite{Allgower2003,Keller1987,Krauskopf2007,Henderson2002}.
\end{remark}

\subsection{Continuation of Periodic Orbits in Conservative Systems}\label{sec:ContinuationOrbits}
In the following, we introduce a modification~$\tilde{\vec{P}}$ to the Poincaré map 
$\vec{P}$ to efficiently trace normal conservative orbits with the numerical continuation methods presented in Section \ref{sec:PCmethods}.
{Furthermore, we show that~$\tilde{\vec{P}}$ and~$\vec{P}$ share the same fixed points.}
Note that that the numerical continuation methods from Section~\ref{sec:PCmethods} are not directly applicable to \eqref{eq:implicitPoincare}, since~$\vec{P}(\vec{x}_0)-\vec{x}_0$ has a square Jacobian matrix.
Hence, we modify the Poincaré map $\tilde{\vec{P}}:\mathcal{A}\times\mathbb{R}\to\mathcal{A}$ with a scalar parameter $\xi$ by
embedding $\Sigma^\mathcal{H}$ in a one-parameter family of dissipative dynamics $\Sigma_\xi=\left(\mathcal{X},\tilde{\mathcal{F}},\mathcal{E},\mathcal{D}\right)$, where
\begin{align}\label{eq:hybridDynamicsNum}
\tilde{\mathcal{F}} := \big\{\dot{\vec{x}}=\underbrace{\vec f_i(\vec x_i)+ \xi\cdot \mr{d}H_i(\vec{x}_i)\Tr}_{=:\tilde{\vec{f}}_i}\big\}_{i=1}^m,
\end{align}
with $H_i\in\mathcal{H}$.
The idea to modify an ODE with a first integral $H_i$ in the form \eqref{eq:hybridDynamicsNum} originates from~\cite{Sepulchre1997}.
Introducing $\xi$ in each phase dynamic~\eqref{eq:hybridDynamicsNum} serves the purpose of consistently influencing the overall hybrid flow~$\vec{x}(t)$, even if the duration~$t_i$ of a specific phase~$i$ approaches zero.
Similar to other modifications of conservative ODEs \cite{Doedel2003,Munoz-Almaraz2003}, the main intention is to resolve inherent state dependencies inflicted by first integrals.
Note that for~$\xi=0$, the conservative dynamics~$\Sigma^\mathcal{H}\equiv\Sigma_0$ are still within the one-parameter family~$\Sigma_\xi$.
In fact, we show that the set of periodic orbits in~$\Sigma_\xi$ is equivalent to periodic orbits of~$\Sigma^\mathcal{H}$.
\begin{theorem}\label{theoremXi0}
The tuple $(\vec{x}_0,\xi)$ is a fixed point of $\tilde{\vec{P}}$ if and only if $\xi=0$ and thus, $\vec{x}_0$ is a fixed point of $\vec{P}$. 
\end{theorem}
\begin{proof}
A proof is provided in Appendix~\ref{sec:appendixProof}.
\end{proof}
Theorem~\ref{theoremXi0} and its proof are based on (Theorem~1~\cite{Sepulchre1997}) and (Lemma~6~\cite{Munoz-Almaraz2003}).
\begin{remark}
Following the approach in \cite{Munoz-Almaraz2003}, we can add more scalars~$\{\xi_j\}_{j=1}^{N_\mathcal{H}}$ to \eqref{eq:hybridDynamicsNum} when~$\Sigma^\mathcal{H}$ has~$N_\mathcal{H}>1$ functionally independent first integrals~$\{H_{i,j}\}_{j=1}^{N_\mathcal{H}}\in\mathcal{H}$ such that
\begin{align*}
\tilde{\vec{f}}_i = \vec f_i(\vec x_i)+ \sum_{j=1}^{N_\mathcal{H}}\xi_j\cdot \mr{d}H_{i,j}(\vec{x}_i)\Tr.
\end{align*}
Based on Lemma~6 in \cite{Munoz-Almaraz2003}, it is straight forward to expand Theorem \ref{theoremXi0} to show that~$\xi_j=0$, for all~$j=1\dots k$, must hold. 
This approach is particularly useful when we want to explore a family of conservative orbits that are not normal.
However, this generates a multi-dimensional continuation problem which is more difficult to handle numerically \cite{Henderson2002}.
Similar to the methods presented in \cite{Ames2005,Ames2006,Colombo2020}, we assume that~$\Sigma^\mathcal{H}$ can be equivalently reduced to a system with~$N_\mathcal{H}=1$.
\end{remark}
With Theorem~\ref{theoremXi0} and the goal to parameterize periodic orbits by level sets $\bar{H}$\footnote{Adding the first integral constraint in \eqref{eq:statebased} is in some sense redundant, since the {set of fixed points of $\tilde{\vec{P}}$} remains unaltered. However, it enables the search for periodic orbits at a specific level set $\bar{H}$.}, we modify \eqref{eq:implicitPoincare} and obtain
\begin{align}\label{eq:statebased}
\vec{r}(\vec{u}):=
    \begin{bmatrix}
        \tilde{\vec{P}}(\vec{x}_0,\xi)-\vec{x}_0\\
        H_1(\vec{x}_0)-\bar{H}
    \end{bmatrix}=\vec{0},
\end{align}
where $\vec{u}\Tr:=[\vec{x}_0\Tr~\xi~\bar{H}]$ and $N=n_1+1$.
\begin{remark}\label{remark:poincareSection}
    The initial state $\vec{x}_0$ in~\eqref{eq:statebased} is not required to lie on the Poincaré section~$\mathcal{A}$. 
    In other words, any~$\vec{x}_0\in\mathcal{X}_1\backslash \mathcal{A}$ in~$\vec{u}$ can still be used to evaluate~$\vec{r}(\vec{u})$.
    However, at a solution point~$\vec{u}\in\vec{r}^{-1}(\vec{0})$, $\vec{x}_0$ will necessarily lie on $\mathcal{A}$ due to the periodicity~$\vec{x}(t^\mathcal{A})=\vec{x}_0$, where $\vec{x}(t^\mathcal{A})\in\mathcal{A}$.
    To consistently ensure that $\vec{x}_0$ originates from the Poincaré section~$\mathcal{A}$, a method involves projecting $\vec{x}_0$ into an $n_1-1$ dimensional coordinate system locally spanning $\mathcal{A}$. Moreover, if the anchor function~$a$ is linear, rendering the Poincaré section~$\mathcal{A}$ a hyperplane, this coordinate projection is globally defined on $\mathcal{X}_1$ (Chapter~9.1~\cite{Leine2004}).
\end{remark}
The map in \eqref{eq:statebased} is now well suited for numerical continuation.
Furthermore, we show how normal conservative orbits are connected to solutions of~\eqref{eq:statebased}.

\begin{theorem}\label{theoremRegularPxi}
Given a normal conservative orbit such that $\vec{x}_0$ is a fixed point of the Poincaré map~$\vec{P}$, $\vec{u}\Tr=\left[\vec{x}_0\Tr~0~H_1(\vec{x}_0)\right]$ is a regular point of \eqref{eq:statebased}. 
\end{theorem}
\begin{proof}
Given that $\vec{x}_0$ lies on a normal conservative orbit, we show in Appendix~\ref{sec:appendixProof} that the Jacobian $[\nicefrac{\partial \tilde{\vec{P}}}{\partial \vec{x}_0}-\vec{I},~\nicefrac{\partial \tilde{\vec{P}}}{\partial \xi}]\vert_{(\vec x_0,0)}$ has full (row) rank.
\end{proof}
Note that the implication in Theorem~\ref{theoremRegularPxi} only goes in one direction.
Hence, a regular point~$\vec{u}^\ast$ of~\eqref{eq:statebased} does not necessarily imply a normal conservative orbit.
In this case, $\vec{u}^\ast$ is called a turning point with respect to the parameterization~$\bar{H}$.
That is, the last entry of the tangent vector~$\vec{\tau}\circ \mr{d}\vec{r}(\vec{u}^\ast)$ vanishes, since there is no local change in the first integral's value.
\subsection{Efficient Implementation: State-based vs. Time-based}\label{sec:StateVsTime}
In the following, we present an alternative definition of the mapping function $\vec{r}(\vec{u})$ as an improvement over \eqref{eq:statebased}. 
This revised definition is more amenable to numerical continuation methods, and we provide detailed insights into its practical implementation.

{\renewcommand{\arraystretch}{1.3}
\begin{table*}[!t]
\centering
 \begin{tabular}{l | l | l } 
 ~ & \emph{state-based} & \emph{time-based} \\ [0.5ex] 
 \hline\hline
 \begin{tabular}{@{}l@{}} Termination of\\ODE solver\end{tabular}
 & \begin{tabular}{@{}l@{}} At events $e_i^{i+1}(\vec{x}_i)=0$\\
 $t_i(\vec{x}_i)$ is implicitly defined \\ by an event activation~\eqref{eq:time2transition}\end{tabular}   & \begin{tabular}{@{}l@{}} At provided time $t_i$\\ $t_i$ is an independent variable\end{tabular} \\ 
 \hline
 Assumptions &  \textit{As1}, \textit{As2}, \textit{As3} & \textit{As1} \\
 \hline
 Restrictions &  \begin{tabular}{@{}l@{}}
  No grazing points ($\mr{d}e_i^{i+1}\vec{f}_i^-\to 0$)\\ No vanishing phase time ($t_i\to 0$)\end{tabular} & \begin{tabular}{@{}l@{}} $t_i$ may not attain its infimum as in \eqref{eq:time2transition}.\\ Proper event activation, i.e., $e_i^{i+1}=0$,\\ is only guaranteed at solution points~$\vec{r}^{-1}(\vec{0})$.\end{tabular}  \\
 \hline
 Dimension of $\vec{r}$ & \begin{tabular}{@{}l@{}} Single shooting: $N=n_1+1$\\Mult. shooting: $N=\sum_{i=1}^{m+1}n_i+1$\end{tabular} &\begin{tabular}{@{}l@{}} Single shooting: $N=n_1+m+2$\\Mult. shooting: $N=\sum_{i=1}^{m+1}n_i+m+2$\end{tabular}  \\
 \hline
 Solution space $\mathcal{S}$ & \begin{tabular}{@{}l@{}} $\mathcal{S}_\text{sb}$ is disconnected in general\\ because of its restrictions.\end{tabular} & \begin{tabular}{@{}l@{}} $\mathcal{S}_\text{tb}$ is the larger solution space in general.\\ It holds $\mathcal{S}_\text{sb}\subseteq\mathcal{S}_\text{tb}$.\end{tabular} 
 \end{tabular}
 \caption{A comparison between state-based and time-based integration and their impact on the map $\vec{r}$ and solution space~$\mathcal{S}=\vec{r}^{-1}(\vec{0})$. Note that the solution spaces become equivalent, i.e., $\mathcal{S}_\text{sb}=\mathcal{S}_\text{tb}$, under full compliance with all assumptions and restrictions.}
 \label{tab:stateVStime}
\end{table*}}

For a revised definition of \eqref{eq:statebased}, we decouple events from the integration process of $\mathcal{F}$ in \eqref{eq:recursion}. 
This brakes the functional dependence of $\{t_i^\mathcal{E}\}_{i=1}^m$ and $t_{m+1}^\mathcal{A}$ with $\{\bar{\vec{x}}_i\}_{i=1}^m$ and $\bar{\vec{x}}_{m+1}$, respectively. 
Notably, the definition in~\eqref{eq:time2transition} no longer holds, leading to a situation where the detection of an event is no longer unique.
Instead, we introduce $m+1$ independent variables~$t_i$ with an additional set of~$m+1$ constraints with events~$e_i^{i+1}=0$ and anchor~$a=0$ to maintain an indirect coupling.
Hence, the new map takes the form
{
\begin{subequations}\label{eq:timebased}
\begin{align}
    \vec{r}(\vec{u}) := \begin{bmatrix}
        \vec{r}_\text{periodicity}(\vec{u})\\
        r_\text{anchor}(\vec{u})\\
        \left\{\begin{matrix}
            \vec{r}_\text{shooting}(\vec{u};i)\\
            r_\text{event}(\vec{u};i)
        \end{matrix}\right\}_{i=1}^m\\
        r_\text{first-integral}(\vec{u})
    \end{bmatrix},
\end{align}
with its abbreviated components
\begin{align}
    \vec{r}_\text{p.}(\vec{u}) &= \vec{\varphi}_{1}(t_{m+1},\bar{\vec{x}}_{m+1};\xi)-\bar{\vec{x}}_1,\\
    r_\text{a.}(\vec{u}) &= a\circ\vec{\varphi}_{1}(t_{m+1},\bar{\vec{x}}_{m+1};\xi),\\
    \vec{r}_\text{s.}(\vec{u};i) &=\vec{\Delta}_i^{i+1}\circ \vec{\varphi}_i(t_i,\bar{\vec{x}}_i;\xi)-\bar{\vec{x}}_{i+1},\label{eq:shooting-step}\\
    r_\text{e.}(\vec{u};i) &= e_i^{i+1}\circ \vec{\varphi}_i(t_i,\bar{\vec{x}}_i;\xi),\\
    r_\text{i.}(\vec{u}) &=H_1(\bar{\vec{x}}_1)-\bar{H},
\end{align}
\end{subequations}}
where {$\vec{u}\Tr:=[t_{m+1}~\bar{\vec{x}}_{m+1}\Tr~t_m~\bar{\vec{x}}_{m}\Tr~\cdots~t_1~\bar{\vec{x}}_{1}\Tr~\xi~\bar{H}]$} and thus, $N= \sum_{i=1}^{m+1}n_i+m+2$.
Please be aware that aside from the time decoupling, the initial phase states $\{\bar{\vec{x}}_i\}_{i=2}^{m+1}$ are intentionally treated as independent decision variables, thus transforming~\eqref{eq:timebased} into a multiple shooting problem, introducing the additional shooting constraints in~\eqref{eq:shooting-step}.
In a single shooting implementation as in~\eqref{eq:statebased}, the initial phase states~$\{\bar{\vec{x}}_i\}_{i=2}^{m+1}$ are recursively computed by~\eqref{eq:recursion} and thus, would not appear in~$\vec{u}$ and~\eqref{eq:timebased}.
The multiple shooting formulation is especially advantageous when it comes to computing the Jacobian $\mr{d}\vec{r}$, as the required sensitivities, denoted as $\mat{\Phi}_i(t):=\nicefrac{\partial \vec{x}_i(t)}{\partial \bar{\vec{x}}_i}$ and $\mat{\Psi}_i(t):=\nicefrac{\partial \vec{x}_i(t)}{\partial \xi}$, can be efficiently computed in parallel by solving matrix ODEs~\cite{Dickinson1976}:
\begin{align}\label{eq:diffState}
\begin{aligned}
    \dot{\mat{\Phi}}_i\left(t\right) &= \frac{\partial \tilde{\vec{f}}_i\left(\vec{x}_i,\xi\right)}{\partial \vec{x}_i}\bigg\vert_{(\vec{x}_{i}(t),\xi)}\mat{\Phi}_i(t),\\
    \mat{\Phi}_i(0) &= \vec{I},
\end{aligned}
\end{align}
\begin{align}\label{eq:diffParam}
\begin{aligned}
    \dot{\mat{\Psi}}_i\left(t\right) &= \frac{\partial \tilde{\vec{f}}_i}{\partial \vec{x}_i}\bigg\vert_{(\vec{x}_{i}(t),\xi)}\mat{\Psi}_i(t)+\frac{\partial \tilde{\vec{f}}_i}{\partial \xi}\bigg\vert_{(\vec{x}_{i}(t),\xi)},\\
    \mat{\Psi}_i(0) &= \vec{0}.
\end{aligned}
\end{align}
For a large number of hybrid phases $m$, the sparse structure of the Jacobian becomes apparent:
\begin{subequations}\label{eq:multishootJacobian}
\begin{flalign}\label{eq:dr}
    \mr{d}\vec{r}=\left[\begin{smallmatrix}
    \vec{R}_{m+1} & \vec{0} & \vec{0} & \dots & \vec{0} & \vec{0} & \vec{C} & \vec{\Xi}_{m+1} & \vec{0}\\[1mm]
     \vec{C} & \vec{R}_{m} & \vec{0} & \dots & \vec{0} & \vec{0} & \vec{0} & \vec{\Xi}_{m} & \vec{0}\\[1mm]
     \vec{0} & \vec{C} & \vec{R}_{m-1} &  & \vec{0} & \vec{0} & \vec{0} & \vec{\Xi}_{m-1} & \vec{0}\\[1mm]
     \vdots & \ddots  & \ddots & \ddots & & \ddots &  & \vdots & \vdots\\[2mm]
     \vec{0} & \vec{0} & \vec{0} & & \vec{C} & \vec{R}_{2} & \vec{0} & \vec{\Xi}_{2} & \vec{0}\\[1mm]
     \vec{0} & \vec{0} & \vec{0} & \dots & \vec{0} & \vec{C} & \vec{R}_{1} & \vec{\Xi}_{1} & \vec{0}\\[1mm]
     \vec{0} & \vec{0} & \vec{0}& \dots & \vec{0} & \vec{0} & \vec{h}_1 & 0 & -1
    \end{smallmatrix}\right],
\end{flalign}
with
\begin{align}\label{eq:Ri}
    \vec{R}_i = \left\{\begin{array}{ll}
        \left[\begin{smallmatrix}
            \tilde{\vec{f}}_{m+1} & \vec{\Phi}_{m+1}\\[1mm]
            \mr{d}a\cdot\tilde{\vec{f}}_{m+1} & \mr{d}a\cdot\vec{\Phi}_{m+1}
        \end{smallmatrix}\right] & ,~i=m+1 \\[3mm]
        \left[\begin{smallmatrix}
            \vec{D}_i^{i+1}\cdot\tilde{\vec{f}}_{i} & \vec{D}_i^{i+1}\cdot\vec{\Phi}_{i}\\[1mm]
            \mr{d}e_i^{i+1}\cdot\tilde{\vec{f}}_{i} & \mr{d}e_i^{i+1}\cdot\vec{\Phi}_{i}
        \end{smallmatrix}\right] & ,~\text{otherwise},
    \end{array}\right.
\end{align}
\begin{align}\label{eq:Xii}
   \vec{\Xi}_i = \left\{\begin{array}{ll}
        \left[\begin{smallmatrix}
            \vec{\Psi}_{m+1}\\[1mm]
            \mr{d}a\cdot\vec{\Psi}_{m+1}
        \end{smallmatrix}\right] & ,~i=m+1 \\[3mm]
        \left[\begin{smallmatrix}
            \vec{D}_i^{i+1}\cdot\vec{\Psi}_{i}\\[1mm]
            \mr{d}e_i^{i+1}\cdot\vec{\Psi}_{i}
        \end{smallmatrix}\right] & ,~\text{otherwise},
    \end{array}\right.
\end{align}
{\begin{align}
   \vec{h}_1 = \left[0~\mr{d}H_1(\bar{\vec{x}}_1)\right],~
    \vec{C} = 
    \begin{bmatrix}
        \vec{0} & -\vec{I}\\
        0 & \vec{0}
    \end{bmatrix}
\end{align}}
\end{subequations}
For compactness, we dropped all function arguments in \eqref{eq:Ri} and \eqref{eq:Xii}. All functions are evaluated at their final phase duration $t_i$ provided with the phase flows $\vec{\varphi}_i(t_i,\bar{\vec{x}}_i;\!\xi)$, which are computed simultaneously by the same numerical integration scheme.
Also note that $\vec{R}_i$, $\vec{C}$ and $\vec{\Xi}_i$ are generally of different size, since the state dimension $n_i$ of a phase varies.

In addition to the subtle distinctions in shooting methods, the primary difference between equations \eqref{eq:statebased} and \eqref{eq:timebased} lies in their handling of simulation times $t_i$. 
To numerically implement \eqref{eq:statebased} and, consequently, to compute solutions for $\Sigma_\xi$, an event-driven ODE solver is required that halts at event times~$t_i^\mathcal{E}(\bar{\vec{x}}_i)$.  
We label the computational process for generating Poincaré maps $\vec{P}$ and $\tilde{\vec{P}}$ as \emph{state-based} integration since both $t_i^\mathcal{E}$ and $t_{m+1}^\mathcal{A}$ explicitly depend on the states~$\vec{x}_i$. 
In contrast, an ODE solver for \eqref{eq:timebased} terminates at pre-specified phase times~$t_i$ independent of any events, and is thus referred to as \emph{time-based} integration.

The advantage of state-based integration in \eqref{eq:statebased} is that it maintains the continuation problem with~$N=n_1+1$ low dimensional and hence, is independent of the number of hybrid phases~$m$ in~$\Sigma^\mathcal{H}$.
Additionally, the monodromy matrix~$\vec{\Phi}_T$ can be readily computed in the Jacobian~$\mr{d}\vec{r}$ and can be utilized for stability analysis. 
However, there are two significant drawbacks to using state-based integration in the context of numerical continuation.
In practice, a solution curve~$\vec{r}^{-1}(\vec{0})$ often connects to points~$\vec{u}^\ast$ \cite{Raff2022a,Gan2018,Rosa2014,Rosa2022},
\begin{itemize}
    \item where grazing occurs, i.e., $\mr{d}e_i^{i+1}\vec{f}_i^-\to 0$
    \item or phase durations vanish, i.e., $t_i\to 0$.
\end{itemize}
Grazing points~$\vec{u}^\ast$ indicate a change in the sign of~$\dot{e}_i^{i+1}$ within~$\Sigma$ causing a solution curve~$\vec{r}^{-1}(\vec{0})$ of \eqref{eq:statebased} to terminate at~$\vec{u}^\ast$ as there are no local solutions of~$\Sigma^\mathcal{H}$ beyond this point that satisfy~$\dot{e}_i^{i+1}<0$.
Moreover, since the saltation matrix~$S_i^{i+1}$ in \eqref{eq:Saltation} is not defined for grazing points, the state-based map \eqref{eq:statebased} is non-differentiable at~$\vec{u}^\ast$.
The second issue arises from the vanishing event time~$t_i$, causing numerical inaccuracies within the respective ODE solver \cite{Shampine2000}. 
The ODE solver typically executes a first integration step to obtain~$\dot{e}_i^{i+1}$ which becomes impractical for~$t_i\to 0$. 

In contrast, time-based integration does not require the construction of the saltation matrix~$\vec{S}_i^{i+1}$ for the computation of the Jacobian~$\mr{d}\vec{r}$ as shown in \eqref{eq:multishootJacobian}. 
In \eqref{eq:timebased}, the direction of event activation~$\dot{e}_i^{i+1}$ and the infimum of phase time~$t_i$ as defined in \eqref{eq:time2transition} are not encoded. 
Consequently, the solution space~$\mathcal{S}$ for this formulation is generally larger than that of the state-based implementation presented in \eqref{eq:statebased}. 
Specifically, the phase trajectory~$\vec{x}_i$ in \eqref{eq:timebased} can also be computed in reverse, for~$t_i<0$. 
Thus, the issues related to grazing and vanishing phase durations (violation of Assumptions~\textit{As2}, \textit{As3}) do not impact the time-based implementation \eqref{eq:timebased}. 
However, it is important to note that a solution to \eqref{eq:timebased} may correspond to an event activation with~$\dot{e}_i^{i+1}\geq 0$, or it may ignore prior activation instances for which~$\dot{e}_i^{i+1}<0$, since~$t_i$ does not correspond to the infimum in~\eqref{eq:time2transition}.
An overview of the main differences between a state-based and time-based implementation is provided in Table~\ref{tab:stateVStime}.
\begin{remark}
    A solution~$\vec{u}$ to~\eqref{eq:timebased}, satisfying the conditions of no grazing and strictly positive time durations, such that it attains its infimum in accordance with \eqref{eq:time2transition}, corresponds to a solution~$\vec{x}_0=\bar{\vec{x}}_1$ of \eqref{eq:statebased}.
    This correspondence is established through the application of the implicit function theorem.
    Additionally, when~$\vec{u}$ is a regular solution of~\eqref{eq:timebased}, it follows that the corresponding solution for~\eqref{eq:statebased} is also regular.
    The same implication holds for turning points and thus, whether~$\vec{x}_0$ corresponds to a normal conservative orbit or not.
\end{remark}

\subsection{Initial Periodic Orbits}
To construct a curve within the set~$\mathcal{S}$ using numerical continuation methods, as described in \eqref{eq:IVP}, it is essential to begin with an initial point~$\vec{u}_0\in\mathcal{S}$. 
When we lack specific knowledge about the hybrid system~$\Sigma^\mathcal{H}$, our goal is to obtain any point~$\vec{u}_0\in\mathcal{S}$ by numerically solving~$\vec{r}(\vec{u})=\vec{0}$ starting from an arbitrary initial guess~$\vec{u}_\text{guess}\in \mathbb{R}^{N+1}$. 
In such cases, the solution space can be reduced by keeping the level set~$\bar{H}$ fixed, which makes the problem more tractable.
Since~$\vec{u}_\text{guess}$ could be significantly distant from~$\mathcal{S}$, a common strategy is to expand the region of convergence by modifying the original problem. One popular approach involves minimizing the sum of squares \cite{Ortega2000} or embedding the problem into a so-called global homotopy \cite{Allgower2003}.

Alternatively, when we possess specific knowledge about the hybrid system~$\Sigma^\mathcal{H}$, we can leverage it to construct an initial $\vec{u}_0$. 
Depending on the characteristics of $\Sigma^\mathcal{H}$, there may be opportunities to obtain local solutions that can be constructed analytically. 
If the system exhibits a time-reversal symmetry, as discussed in \cite{Colombo2020} and \cite{Hyon2005}, a constructive approach to finding a starting point $\vec{u}_0$ becomes available. 
Furthermore, for scenarios where our understanding of $\Sigma^\mathcal{H}$ is limited, we can turn to the time-based implementation described in \eqref{eq:timebased}. 
This approach allows us to explicitly exploit analytic grazing solutions, such as equilibria \cite{Rosa2014}, or hybrid dynamics that evolve with zero time duration, i.e., $t_i=0$. 
We will demonstrate the application of this method in constructing initial points for all four examples presented in Section~\ref{sec:Examples}.

\section{Examples from Mechanics}
\label{sec:Examples}
To demonstrate the numerical computation and continuation of normal conservative orbits, we discuss four examples of cHDSs: a bouncing ball, a rocking block, a bouncing rod and a one-legged hopper based on the spring-loaded-inverted-pendulum (SLIP).
They are all based on energetically conservative mechanical systems with lossless collisions and their hybrid first integral~$\mathcal{H}$ consistently represents mechanical energy.
To facilitate comparisons between them, each systems' state and parameter values have been normalized with respect to total mass~$m_\mr{o}$, characteristic length $l_\mr{o}$ and gravity $g$.
The hybrid dynamics~$\Sigma^\mathcal{H}$ of these systems are smooth in all components and thus, $\vec{r}\in C^\infty$.
We specifically highlight the time-based implementation of~$\vec{r}$ for these systems due to its ability to explore non-physical behavior that nonetheless smoothly connects to normal conservative orbits.
The four systems are shown in Figure~\ref{fig:examples}.

We employ a higher-order variable step size integrator (explicit Runge-Kutta (4,5) with error tolerances~$10^{-7}$) to trade-off accuracy and speed in the numerical computation of the map~$\vec{r}(\vec{u})$~\eqref{eq:timebased}.
All four examples are implemented in both MATLAB and Julia, and the corresponding code can be found at \href{https://github.com/raffmax/ConservativeHybridDynamicalSystems}{GitHub}\footnote{\tiny\url{https://github.com/raffmax/ConservativeHybridDynamicalSystems}}.
\begin{figure}[t]
    \centering
    \includegraphics{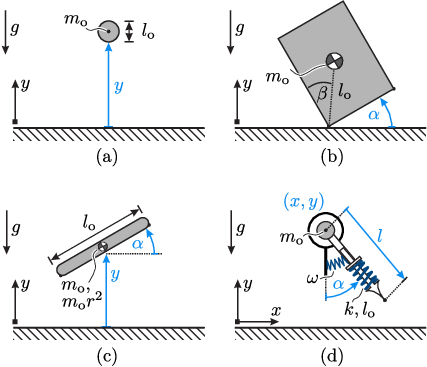}
    \caption{Shown are four conservative mechanical systems with impulsive dynamics due to lossless collisions with the ground:
    (a)~bouncing ball, (b)~rocking block, (c)~bouncing rod and (d)~one-legged hopper.
    The system's configuration variables are shown in blue.}
    \label{fig:examples}
\end{figure}

\subsection{The Bouncing Ball}
The cHDS~$\Sigma^\mathcal{H}$ of the bouncing ball (mass~$m_\mr{o}$, diameter~$l_\mr{o}$), shown in Fig.~\ref{fig:examples}a, is represented as:

\begin{subequations}\label{eq:HDSball}
    \begin{align}
\mathcal{X}_1&=\left\{\vec{x}_1=\begin{bmatrix}
    y\\\dot{y}
\end{bmatrix}~\Bigg\vert~y\in\mathbb{R},\dot{y}\in\mathbb{R}\right\},\\
\vec{f}_1 &=\begin{bmatrix}\dot{y} & -g\end{bmatrix}\Tr,\\
\mathcal{E}_1^1&=\left\{\vec{x}_1=\begin{bmatrix}
    y\\\dot{y}
\end{bmatrix}~\Bigg\vert~\begin{matrix}e_1^1=y=0,\\
    \dot e_1^1=\dot{y}<0
\end{matrix}\right\},\label{eq:Eball}\\
\vec{\Delta}_1^1&=\begin{bmatrix}
    y & -\dot{y}
\end{bmatrix}\Tr,\label{eq:DelatBall}\\
H_1&=\frac{1}{2}m_\mr{o}\dot{y}^2+m_\mr{o}gy.\label{eq:Hball}
\end{align}
\end{subequations}
These dynamics are limited to a single phase ($m=1$) and represent a free-falling object traversing the phase sequence $1\to 1$. 
The hybrid phase transition occurs when the ball touches the ground and the state returns to $\mathcal{X}_1$ via a fully elastic collision to complete the recurrent motion as illustrated in Fig.~\ref{fig:ballrockFrame}.
\begin{figure}[t]
    \centering
    \includegraphics{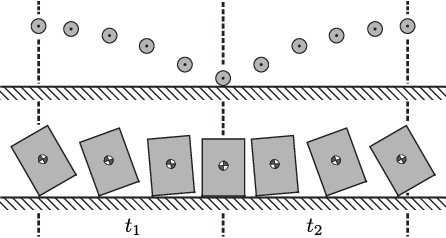}
    \caption{Key frames illustrating a normal conservative orbit of the bouncing ball and the rocking block. In the rocking block example, the system's symmetric motion is exploited, allowing us to compute only its pivoting motion about the left corner.
    }
    \label{fig:ballrockFrame}
\end{figure}
The total energy of the ball in~\eqref{eq:Hball} serves as the hybrid first integral~$\mathcal{H}$.
Given the simplicity of the example, the hybrid flow~$\vec{\varphi}$ of the cHDS~$\Sigma^\mathcal{H}$ of \eqref{eq:HDSball} can be solved analytically according to the recursion in~\eqref{eq:hybridFlow}, which helps illustrate the concepts and symbols introduced in Section~\ref{sec:HDS}.
\begin{subequations}
The hybrid flow emanates from an arbitrary initial point
\begin{align}
\bar{\vec{x}}_1 &=\vec{x}_0=\begin{bmatrix}
    y_{0} & \dot{y}_0
\end{bmatrix}\Tr\in\mathcal{X}_1,
\end{align}
where $\vec{x}_0\in\mathcal{X}_0$ must be further constrained to produce physically realizable recurrent hybrid trajectories that satisfy Assumptions~\textit{As1}-\textit{As3}. Specifically, the initial height must satisfy $y_0\in \mathbb{R}_{>0}$ such that $e_1^1(\vec{x}_0)>0$.
Given this, the hybrid flow progresses through the first phase with
\begin{align}
\vec{\varphi}_1\left(t_1,\bar{\vec{x}}_1\right)
&=\begin{bmatrix}
    -\frac{1}{2}gt_1^2 +\dot{y}_0 t_1 +y_0\\ -gt_1+\dot{y}_0
\end{bmatrix},
\end{align}
until it reaches the event manifold~$\mathcal{E}_1^1$ at time
\begin{align}
    t_1^\mathcal{E}(\bar{\vec{x}}_1)=\frac{\dot{y}_0+\sqrt{\dot{y}_0^2+2y_0g}}{g},\label{eq:t1Ball}
\end{align}
with $e_1^1\circ\vec{\varphi}_1(t_1^\mathcal{E},\bar{\vec{x}}_1)=0$.
The second part of the hybrid flow is initiated at the image of the reset map~$\vec{\Delta}_1^1$ with
\begin{align}
&\begin{aligned}
\bar{\vec{x}}_2 &=\vec{\Delta}_1^1\circ\vec{\varphi}_1\left(t_1^{\mathcal{E}}(\bar{\vec{x}}_1),\bar{\vec{x}}_1\right)\\
&=\begin{bmatrix}
    -\frac{1}{2}gt_1^\mathcal{E}(\bar{\vec{x}}_1)^2 +\dot{y}_0 t_1^\mathcal{E}(\bar{\vec{x}}_1) +y_0\\ -\big(-gt_1^\mathcal{E}(\bar{\vec{x}}_1)+\dot{y}_0\big)
\end{bmatrix},
\end{aligned}\\
&\overset{\eqref{eq:t1Ball}}{\Rightarrow}\bar{\vec{x}}_2=\begin{bmatrix}
y_{2}\\ \dot{y}_2
\end{bmatrix}
=\begin{bmatrix}
    0\\ \sqrt{\dot{y}_0^2+2y_0g}
\end{bmatrix}.\label{eq:x2bar}
\end{align}
Hence, with the abbreviation $t_2=t-t_1^\mathcal{E}(\bar{\vec{x}}_1)$, the overall hybrid flow at time~$t\in\mathcal{T}$ is defined as
\begin{align}\label{eq:hybridFlowBall}
\begin{aligned}
    \vec \varphi(t,\vec x_0) = \vec{\varphi}_1\left(t_2,\bar{\vec{x}}_{2}\right)= \begin{bmatrix}
    -\frac{1}{2}gt_2^2 +\dot{y}_2 t_2 \\ -gt_2+\dot{y}_2
\end{bmatrix}. %
\end{aligned}
\end{align}
The time $t$ can be chosen from the final time interval $\mathcal{T} =t_1^{\mathcal{E}}(\bar{\vec{x}}_1)+[0,t_2^{\mathcal{E}}(\bar{\vec{x}}_2)]$,
in which the time~$t_2^{\mathcal{E}}(\bar{\vec{x}}_2)$ with
\begin{align}
&t_2^{\mathcal{E}}(\bar{\vec{x}}_2) = \frac{2}{g}\dot{y}_2\overset{\eqref{eq:x2bar}}{=}\frac{2}{g}\sqrt{\dot{y}_0^2+2y_0g}
\end{align}
would be the time instance of the next collision with the ground.
\end{subequations}
To eliminate time~$t$ and identify periodic orbits in~\eqref{eq:hybridFlowBall}, we define the Poincaré section
\begin{align}
\mathcal{A}=\left\{\vec{x}_1=\begin{bmatrix}
    y\\\dot{y}
\end{bmatrix}~\Bigg\vert~\begin{matrix}a=\dot{y}=0,\\
    \dot a=-g<0
\end{matrix}\right\},\label{eq:Aball}
\end{align}
which corresponds to the ball being instantaneously at rest.
\begin{subequations}\label{eq:Pball}
This implies the time-to-anchor function
$t_2^\mathcal{A}(\bar{\vec{x}}_2)=\nicefrac{\dot{{y}}_2}{g}$,
leading to the overall accumulated time:
\begin{align}
\begin{aligned}
t^\mathcal{A}(\vec{x}_0)&=t_1^\mathcal{E}(\bar{\vec{x}}_1)+t_2^\mathcal{A}(\bar{\vec{x}}_2)=\frac{\dot{y}_0+2\dot{y}_2}{g}\\
&=\frac{\dot{y}_0+2\sqrt{\dot{y}_0^2+2y_0g}}{g}.
\end{aligned}
\end{align}
With that, the Poincaré map is given as
\begin{align}
\vec{P}(\vec{x}_0)=\vec{x}(t^\mathcal{A})=\begin{bmatrix}
    \frac{\dot{y}_0^2}{2g}+y_0\\0
\end{bmatrix},\quad\vec{x}_0=\begin{bmatrix}
    y_{0} \\ \dot{y}_0
\end{bmatrix}.
\end{align}
\end{subequations}
Note, although $\vec{P}(\vec{x}_0)\in\mathcal{A}$, the initial state~$\vec{x}_0$ still has to be chosen from the Poincaré section~$\mathcal{A}$~(Remark~\ref{remark:poincareSection}).
This requires~$\dot{y}_0=0$ and simplifies~\eqref{eq:Pball} to
\begin{subequations}\label{eq:PballPeriodic}
\begin{align}
&T=t^\mathcal{A}\left(\begin{bmatrix}y_0 & 0\end{bmatrix}\Tr\right)=2\sqrt{2\frac{y_0}{g}},\\
&\vec{P}\left(\begin{bmatrix}y_0 \\ 0\end{bmatrix}\right)=\begin{bmatrix}y_0 \\ 0\end{bmatrix}.
\end{align}
\end{subequations}
Hence, any $\vec{x}_0=[y_0 ~ 0 ]\Tr\in\mathcal{A}$ with $y_0>0$ yields a fixed point of the Poincaré map~$\vec{P}$ and a normal conservative orbit of \eqref{eq:HDSball}, which also fulfills Assumption~\textit{As3}.
That is, these fixed points are not isolated, as anticipated by Theorem~\ref{thm:continuance} based on the existence of a hybrid first integral~$\mathcal{H}$.
A projection of this family of orbits is depicted by the solid blue line in Fig.~\ref{fig:HoverT}a.

For the numerical exploration of periodic orbits, as discussed in Section~\ref{sec:ContinuationOrbits}, we consider the time-based implementation~\eqref{eq:timebased}:
\begin{subequations}\label{eq:rball}
\begin{align}
    \vec{r}(\vec{u})&= \begin{bmatrix}
        \vec{\varphi}_{1}(t_{2},\bar{\vec{x}}_{2};\xi)-\bar{\vec{x}}_1\\
        a\circ\vec{\varphi}_{1}(t_{2},\bar{\vec{x}}_{2};\xi)\\
            \mat{\Delta}_1^1 \circ\vec{\varphi}_1(t_1,\bar{\vec{x}}_1;\xi)-\bar{\vec{x}}_{2}\\
           e_1^1 \circ\vec{\varphi}_1(t_1,\bar{\vec{x}}_1;\xi)\\
        H_1(\bar{\vec{x}}_1)-\bar{H}
    \end{bmatrix},\label{eq:timebasedBall}\\
    \vec{u}&=\begin{bmatrix}
        t_{2} & \bar{\vec{x}}_{2}\Tr & t_1 & \bar{\vec{x}}_{1}\Tr & \xi & \bar{H}
    \end{bmatrix}\Tr,
\end{align} 
\end{subequations}
where the phase flows~$\vec{\varphi}_1$ solve the modified ODE defined by the vector field in \eqref{eq:hybridDynamicsNum}:
\begin{align}\label{eq:fTildeBall}
    \tilde{\vec{f}}_1(\vec{x}_1,\xi)=\begin{bmatrix}
        \dot{y}+\xi m_\mr{o}g\\-g+\xi m_\mr{o}\dot{y}
    \end{bmatrix},\quad \vec{x}_1 =\begin{bmatrix}
        y\\ \dot{y}
    \end{bmatrix}.
\end{align}
Note that for $\xi=0$, it holds~$\tilde{\vec{f}}_1=\vec{f}_1$ and in particular, with~$\bar{H}>0$, the solution spaces of periodic orbits of~\eqref{eq:rball} and \eqref{eq:PballPeriodic} are equivalent.
With the map~$\vec{r}$ and its Jacobian~$\mr{d}\vec{r}$, we explored the solution space~$\mathcal{S}=\vec{r}^{-1}(\vec{0})$ utilizing a predictor-corrector method (Section~\ref{sec:PCmethods}).

The use of numerical continuation methods necessitates an initial point~$\vec{u}_0\in\mathcal{S}$.
With $m=1$, an appropriate initial point for the continuation of periodic solutions is the equilibrium state of $\Sigma^\mathcal{H}$ at energy level~$\bar{H}=0~m_\mr{o}gl_\mr{o}$ and thus, the trivial vector $\vec{u}_0=\vec{0}$.
Notably, this initial point does not fulfill Assumption~\textit{As3}, which, however, does not pose any issues for the time-based implementation~\eqref{eq:timebased}.
Upon closer inspection, one determines that this $\vec{u}_0$ marks a simple bifurcation~(SB) point of $\vec{r}$ with perpendicular tangent vectors $\vec{\tau}_1$ and~$\vec{\tau}_2$, as they are depicted in Fig.~\ref{fig:HoverT}a.
We construct these tangent vectors through a Lyapunov-Schmidt reduction.
Since the Jacobian~$\mr{d}\vec{r}(\vec{u}_0)$ exhibits a vanishing gradient with respect to~$\xi$, the tangent vector~$\vec{\tau}_1$ exclusively aligns with the direction of~$\xi$.
The corresponding projection of points from~$\mathcal{S}$ corresponds entirely to instances in which the ball is completely at rest (orange line in Fig.~\ref{fig:HoverT}a).
Such solutions exist for arbitrary values of~$\xi$.

By following the tangent vector $\vec{\tau}_2$ in the direction of increasing time durations $t_1$ and $t_2$, a normal conservative orbit arises, that corresponds to the fixed points of \eqref{eq:PballPeriodic} (solid blue line).
For negative time durations $t_1<0$ and $t_2<0$ (illustrated by the dashed blue line), the phase dynamics traverse in reverse time. 
While this constitutes a non-physical solution, it is possible to solve the phase dynamics in~\eqref{eq:fTildeBall} in reversed time, by modifying the vector field such that $\dot{\vec{x}}_1=-\tilde{\vec{f}}_1(\vec{x}_1,\xi)$ and solving the corresponding ODE in forward time~$t\in[0,-t_i]$ for $t_i<0$.
Note that only a point of the solid blue line corresponds to a normal conservative orbit and thus, to a solution of~$\Sigma^\mathcal{H}$.
Since the normal conservative orbits in this example can be computed analytically, the relationship between the period time and the first integral's values can be expressed in closed form as~$T=\sqrt{\bar{H}~\nicefrac{8}{m_\mr{o}g^2}}$. For the non-physical periodic orbits with reversed time (dashed blue line) this relationship is expressed by~$T=-\sqrt{\bar{H}~\nicefrac{8}{m_\mr{o}g^2}}$.
\begin{figure*}[t]
    \centering
        \centering
    \includegraphics{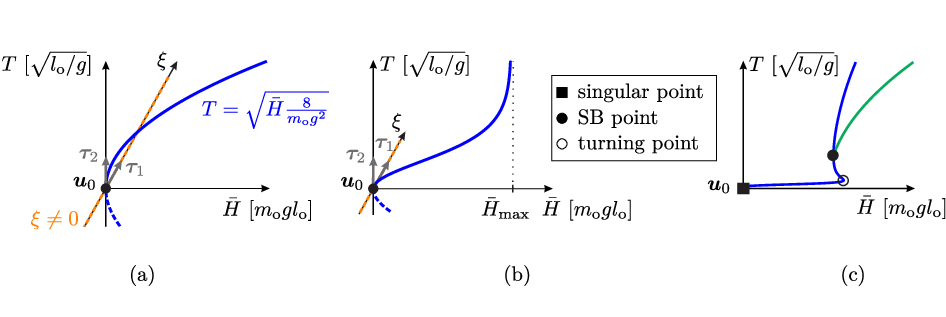}
    \vspace{-30pt} %
    \caption{Projection of points from~$\mathcal{S}$ of (a) the bouncing ball, (b) the rocking block (b), and (c) the bouncing rod. Solid blue lines indicate normal conservative orbits originating from $\vec{u}_0=\vec{0}$. 
    The simplicity of the dynamics of the bouncing ball permits an analytic relationship between $\bar{H}$ and $T$. 
    For the rocking block, $\bar{H}$ is bounded by a homoclinic orbit at $\bar{H}_\text{max}$ as $T$ tends to infinity. 
    For the bouncing rod, $\vec{u}_0$ constitutes a non-simple bifurcation point.
    The family of conservative orbits further features turning points in $\bar{H}$ and a simple bifurcation (SB) point. Projections of the bifurcating solutions are shown by a solid green line, with further details provided in Fig.~\ref{fig:bouncingrod}.}
    \label{fig:HoverT}
\end{figure*}
\subsection{The Rocking Block}
Similar to the example in \cite{Charalampakis2022}, we examined a rectangular block of diagonal size $2l_\mr{o}$, mass $m_\mr{o}$ and so called block slenderness $\beta=0.3~\mr{rad}$, as depicted in Fig.~\ref{fig:examples}b.
The block rocks back and forth between its left and right corners, experiencing lossless collisions during each movement.
A detailed description of this motion can be found in \cite{Housner1963}. 
Due to the block's inherent symmetry, its continuous dynamics during pivoting motion around either corner are equivalent.
Key frames illustrating this symmetric motion are shown in Fig.~\ref{fig:ballrockFrame}.
Consequently, with $m=1$, we define the block's dynamics as
\begin{subequations}\label{eq:HDSblock}
    \begin{align}
\mathcal{X}_1&=\left\{\vec{x}_1=\begin{bmatrix}
    \alpha\\ \dot{\alpha}
\end{bmatrix}:\alpha\in\mathbb{R},\dot{\alpha}\in\mathbb{R}\right\},\\
\vec{f}_1 &= \begin{bmatrix}\dot{\alpha}\\-\frac{3g}{4l_\mr{o}}\sin(\beta-\alpha)\end{bmatrix},\\
\mathcal{E}_1^1&=\left\{\vec{x}_1=\begin{bmatrix}
    \alpha\\\dot{\alpha}
\end{bmatrix}~\Bigg\vert~\begin{matrix}e_1^1=\alpha=0,\\
    \dot e_1^1=\dot{\alpha}<0
\end{matrix}\right\},\\
\vec{\Delta}_1^1&=\begin{bmatrix}
    \alpha & -\dot{\alpha}
\end{bmatrix}\Tr,\\
H_1&=\frac{2}{3}m_\mr{o} l_\mr{o}^2\dot{\alpha}^2+\left(\cos(\alpha-\beta)-\cos(\beta)\right)m_\mr{o}l_\mr{o}g,\label{eq:Hblock}
\end{align}
\end{subequations}
where the angle~$\alpha$ characterizes its rotation about the left corner and the total energy of the block in~\eqref{eq:Hblock} serves as the hybrid first integral~$\mathcal{H}$.
Note that, similar to the bouncing ball example, $H_1$ is defined such that its potential energy is bounded by the zero level set as $\alpha \to 0$.
We further impose the condition~$\alpha < \beta$, as the slenderness ratio~$\beta$ defines the upright unstable equilibrium of the block.
As a result, to ensure that the recurrent hybrid trajectories~$\vec{x}(t)$ are physically realizable and comply with Assumptions~\textit{As1}-\textit{As3}, the initial state domain~$\mathcal{X}_0$ must only include configurations where~$\alpha$ lies within the range~$(0,\beta)$. As before, this ensures~$e_1^1(\vec{x}_0)>0$.

We define the Poincaré section $\mathcal{A}$ based on the anchor function
\begin{align}
    a(\vec{x}_1)=\dot{\alpha},\quad \vec{x}_1=\begin{bmatrix}
        \alpha & \dot{\alpha}
    \end{bmatrix}\Tr,
\end{align} and formulate a time-based problem~\eqref{eq:timebased} of size~$N=7$, aligning with the structure of the preceding bouncing ball example.
Utilizing the equilibrium state~$\vec{u}_0=\vec{0}$ of~$\Sigma^\mathcal{H}$ as the starting point for the continuation process, the solid blue line in Fig.~\ref{fig:HoverT}b once again depicts a family of normal conservative orbits.
Similar to the bouncing ball example, any point~$\vec{x}_0=[\alpha_0~0]\Tr\in\mathcal{A}$ with~$\alpha_0\in(0,\beta)$ corresponds to a fixed point of~$\vec{P}$. 
However, since $\vec{f}_1$ defines the dynamics of a conservative pendulum, the normal conservative orbits $\vec{x}(t)$ of $\Sigma^\mathcal{H}$ cannot be expressed in terms of elementary functions and instead require elliptic functions~\cite{Charalampakis2022}.
In this context, a noteworthy observation is the system's energy upper limit~$\bar{H}_\text{max}$ at $\alpha=\beta$ as the system approaches a homoclinic orbit for $T\to\infty$~\cite{Charalampakis2022}.
Additional solution curves in $\mathcal{S}$ with negative times and non-zero values of $\xi$ exist analogous to the bouncing ball example.

\subsection{The Bouncing Rod}
\begin{figure*}[t]
    \centering
    \includegraphics{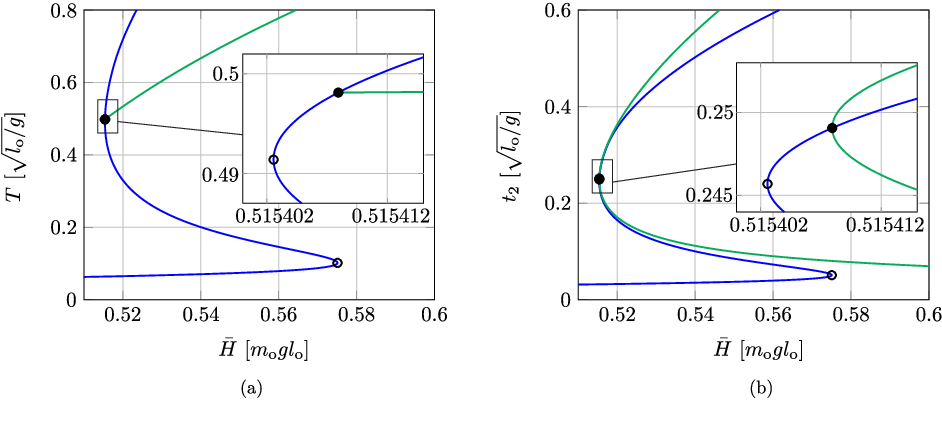}
    \vspace{-30pt} %
    \caption{Projection of orbits of the bouncing rod system, highlighting the turning points as circles and simple bifurcation (SB) points as dots. The solid blue line denotes symmetric orbits with $t_1+t_3=t_2$, while the solid green line illustrates a non-symmetric orbit that originates from an SB point.
    Shown in (a) is a detailed and rotated view of Fig.~\ref{fig:HoverT}c, that highlights the second turning point and the SB point.
    A different projection (using the time of the second flight phase $t_2$) is shown in (b).
    }
    \label{fig:bouncingrod}
\end{figure*}
\begin{figure}[t]
    \centering
    \includegraphics{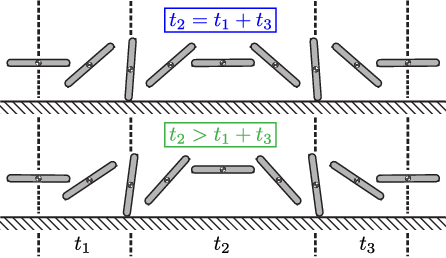}
    \caption{Key frames illustrating two normal conservative orbits of the bouncing rod at period time~$T = 0.8~\sqrt{l_\mr{o}/g}$. The frames compare a symmetric orbit ($t_2 = t_1 + t_3$) with a non-symmetric orbit~($t_2 > t_1 + t_3$). Unlike the non-symmetric case, the symmetric motion reaches the same maximum height twice.
    }
    \label{fig:rodKeyFrame}
\end{figure}
We consider a rod of length $l_\mr{o}$, mass $m_\mr{o}$ and a radius of gyration $r=\nicefrac{l_\mr{o}}{100}$.
The rod's rotation is described by the angle~$\alpha\in (-\pi,\pi)$ and the distance of its center of mass to the ground by $y\in\mathbb{R}$.
Hence, the configuration~$\vec{q}\Tr=[y~\alpha]$ is defined in the configuration space $\mathcal{Q}=\mathbb{R}\times(-\pi,\pi)$.
The cHDS~$\Sigma^\mathcal{H}$ of the bouncing rod is of phase dimension $m=2$ and and is expressed as:
\begin{subequations}\label{eq:HDSrod}
    \begin{align}
&\mathcal{X}_1=\mathcal{X}_2=\left\{\vec{x}_1=\begin{bmatrix} y\\ \alpha \\ \dot{y}\\ \dot{\alpha}
\end{bmatrix}:
    \begin{array}{l}
        [y ~ \alpha]\Tr \in \mathcal{Q},\\\relax
        [\dot{y}~\dot{\alpha}]\Tr \in T_{\vec{q}}\mathcal{Q}
    \end{array}\right\},\\
&\vec{f}_1=\vec{f}_2=\begin{bmatrix}
        \dot{y} & \dot{\alpha} & -g & 0
    \end{bmatrix}\Tr,\label{eq:frod}\\
&e_1^2 = y- \frac{1}{2}l_\mr{o} \sin(\alpha),~~e_2^1 = y+ \frac{1}{2}l_\mr{o} \sin(\alpha),\label{eq:eLeR}\\
    &\vec{\Delta}_1^2 = \vec{x}_1-\frac{4}{l_\mr{o}^2c_\alpha^2+4r^2}\begin{bmatrix}
        0\\0\\ 2r^2\dot{y}-l_\mr{o}r^2c_\alpha\dot{\alpha}\\ \frac{1}{2}l_\mr{o}^2c_\alpha^2\dot{\alpha}-l_\mr{o}c_\alpha\dot{y}
    \end{bmatrix},\label{eq:jumpLeft}\\
         &\vec{\Delta}_2^1 = \vec{x}_2-\frac{4}{l_\mr{o}^2 c_\alpha^2+4r^2}\begin{bmatrix}
        0\\0\\ 2r^2\dot{y}+l_\mr{o}r^2c_\alpha\dot{\alpha}\\ \frac{1}{2}l_\mr{o}^2c_\alpha^2\dot{\alpha}+l_\mr{o}c_\alpha\dot{y}
    \end{bmatrix}\label{eq:jumpRight},\\
&H_{1,1} = H_{2,1} = \frac{1}{2}m_\mr{o}(\dot{y}^2+r^2\dot{\alpha}^2)+m_\mr{o}gy,\label{eq:energyRod}
\end{align}
\end{subequations}
where $c_\alpha:=\cos(\alpha)$. 
The rod possesses contact points on its left and right corners. 
Contacts are treated as fully elastic, resulting in a flight phase between collisions. The event functions~$e_1^2$ and~$e_2^1$ denote when the left and right corners impact the ground, respectively. 
Therefore, the dynamics in~\eqref{eq:HDSrod} represent a left-right contact sequence with~$1\to 2\to 1$. The total energy of the rod in both phases in~\eqref{eq:energyRod} serves as the hybrid first integral~$\mathcal{H}$. It is notable that the angle~$\alpha$ does not appear in~\eqref{eq:energyRod}, indicating another first integral denoted as:
\begin{align}\label{eq:angMomentum}
H_{1,2}=H_{2,2}=m_\mr{o}r^2\dot{\alpha},
\end{align}
referred to as the angular momentum~\cite{Gorni2021}.
However, the discrete maps~$\vec{\Delta}_1^2$ and $\vec{\Delta}_2^1$ instantaneously change the value of~\eqref{eq:angMomentum} for any $\alpha\in (-\pi,\pi)$ in the event sets $\mathcal{E}_1^2$ and $\mathcal{E}_2^1$, respectively.
As~$\{H_{i,2}\}_{i=1}^2$ does not meet Definition~\textit{Df2}\footnote{It can be shown that~$\vec{f}_1$ in \eqref{eq:frod} possesses three independent first integrals, yet only the energy in \eqref{eq:energyRod} satisfies Definition~\textit{Df2}.}, only the system's energy~$\mathcal{H}=\{H_{1,1},H_{2,1}\}$ is considered a hybrid first integral of the cHDS~$\Sigma^\mathcal{H}$.
As with the previous examples, trajectories originating from~$\mathcal{X}_0$ must be physically realizable to satisfy Assumptions~\textit{As1}-\textit{As3}. Hence, any~$\vec{x}_0 \in \mathcal{X}_1$ must fulfill~$e_1^2(\vec{x}_0)>0$.

The anchor is positioned at apex transit such that
\begin{align}
    a(\vec{x}_1)=\dot{y},\quad \vec{x}_1=\begin{bmatrix}
        y& \alpha & \dot{y}&\dot{\alpha}
    \end{bmatrix}\Tr.
\end{align}
Unlike the previous examples, not every point on the Poincaré section $\mathcal{A}$ corresponds to a periodic orbit.
Additionally, while the system's equilibrium~$\vec{u}_0=\vec{0}$ remains a singular point of the time-based implementation~\eqref{eq:timebased}, it does not constitute an SB point.
Given $m=2$ hybrid phases, the problem is of size $N=17$ with the decision vector~$\vec{u}\Tr=[t_{3}~\bar{\vec{x}}_{3}\Tr~t_{2}~\bar{\vec{x}}_{2}\Tr~t_1~\bar{\vec{x}}_{1}\Tr~\xi~\bar{H}]$.
To initiate branching from $\vec{u}_0=\vec{0}$ into a normal conservative orbit, we construct the tangent vector directing towards a symmetric orbit with $t_1+t_3=t_2$. 
The solid blue line in Fig.~\ref{fig:HoverT}c illustrates this family of symmetric normal conservative orbits.
However, these orbits cease to be normal at two turning and an SB point.
At the SB point, the bifurcating orbits, depicted by solid green lines, break the previously mentioned symmetry.
As a result, the SB point creates two additional 1D families of non-symmetric, normal conservative orbits. 
These orbits exhibit a phase duration~$t_2$ that is either shorter or longer than $t_1+t_3$, corresponding to different heights of the rod at the two apexes within a periodic orbit.
Fig.~\ref{fig:bouncingrod} zooms in on the region of special orbits, revealing the two turning points and a projection of the bifurcating orbits.
{The key frames in Fig.~\ref{fig:rodKeyFrame} further illustrate the differences between symmetric and non-symmetric orbits. 
They also emphasize the system's stiff dynamics due to the low radius of gyration~$r$. 
Consequently, while the vertical position~$y$ changes only slightly, the rod undergoes rapid rotation with high angular velocity~$\dot{\alpha}$.}

\subsection{The Spring-Loaded-Inverted-Pendulum}
\begin{figure*}[t]
    \centering
    \includegraphics{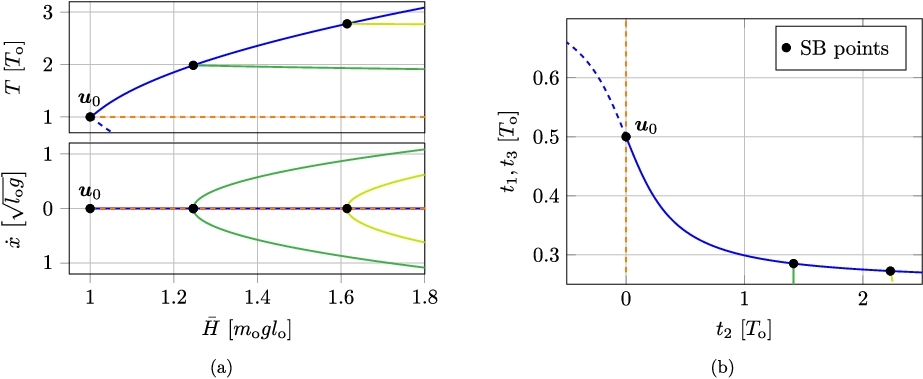}
    \vspace{-5pt} %
    \caption{{}Various projections of points from $\mathcal{S}$ of the SLIP model.  
    Shown in (a) are period time $T$ and forward speed $\dot{x}$ during flight over energy level~$\bar{H}$.
    {(b) shows the stance durations, $t_1$ and $t_3$, relative to the flight time~$t_2$. The phase times after and before the nadir transit, $t_1$ and $t_3$ respectively, appear identical in this projection.}
    The period time~$T_\mr{o}=2\pi \sqrt{m_\mr{o}/k}$ illustrates a vertical harmonic oscillation with no flight time~($t_2=0$), depicted by the dashed orange lines. 
    For positive flight times~$t_2>0$, the solid blue line corresponds to hopping-in-place~($\dot{x}=0$), while negative flight times~$t_2<0$ (dashed blue line) are non-physical.
    Simple bifurcations can be found at higher energy levels $\bar{H}$, depicting the bifurcating solutions of stationary hopping (solid blue line) into forward/backward hopping (solid green lines). 
    }
    \label{fig:slip}
\end{figure*}
\begin{figure*}[t]
    \centering
    \includegraphics{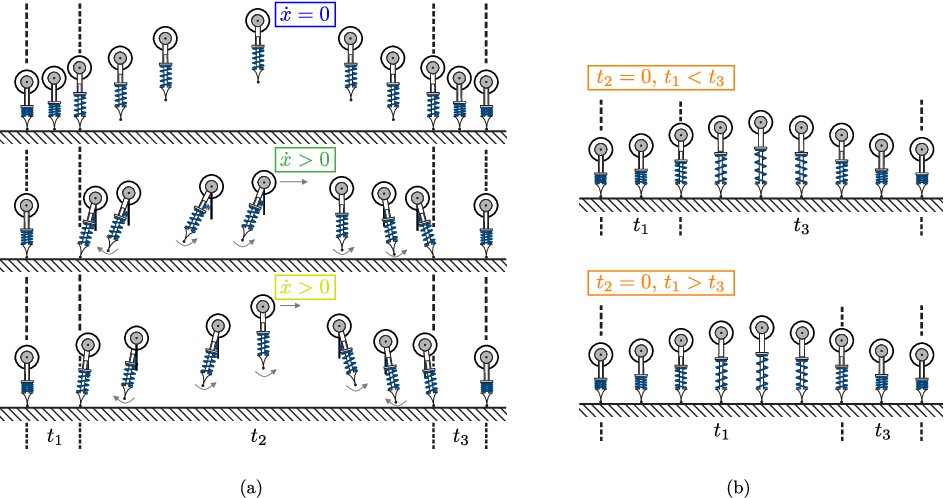}
    \vspace{-15pt} %
    \caption{{Key frames illustrating various periodic orbits of the hopper example. Shown in (a) are three distinct normal conservative orbits of the hopper at energy level~$\bar{H}=1.8~m_\mr{o}gl_\mr{o}$. Additional arrows are provided to visualize linear and angular velocities within the motion.
    In contrast, (b) depicts harmonic in-place oscillations with vanishing flight time~$t_2=0$, which are considered non-physical as they violate Assumptions~\textit{As1}-\textit{As3}.}
    }
    \label{fig:slipKeyFrame}
\end{figure*}
The SLIP-like hopper with swing-leg dynamics in Fig.~\ref{fig:examples}d {is adapted} from~\cite{Gan2018,Raff2022a}.
It is an intriguing example of legged locomotion, featuring several bifurcations of periodic motions (gaits). Unlike the previous examples, it also involves varying phase domains when minimal coordinates are used.
The hopper consists of a single point mass~$m_\mr{o}$ attached to a mass-less leg with a natural length~$l_\mr{o}$ and stiffness~$k = 40\nicefrac{m_\mr{o}g}{l_\mr{o}}$.
{When the leg is not in contact with the ground, its orientation oscillates harmonically with a frequency~$\omega = \sqrt{5\nicefrac{g}{l_\mr{o}}}$.}
{As elaborated in \cite{Gan2018,Raff2022a}, the squared frequency~$\omega^2$ represents the ratio of hip spring stiffness to leg inertia. 
Throughout the modeling process, this ratio remains a constant finite value, even as leg inertia approaches zero, ensuring energy conservation when the leg contacts the ground.
This alteration of ground contact is characterized by a stance phase (S) and a flight phase (F) of the hopper.}
It establishes a cHDS~$\Sigma^\mathcal{H}$ with two hybrid phases ($m=2$):
{\begin{subequations}\label{eq:HDSslip}
\begin{align}
    &\mathcal{X}_\text{S}=\left\{\vec{x}_\text{S}=\begin{bmatrix}
        \alpha \\ l \\ \dot{\alpha} \\ \dot{l}
    \end{bmatrix}:
    \begin{array}{l}
        [\alpha~l]\Tr \in \mathcal{Q}_\text{S},\\\relax
        [\dot{\alpha}~\dot{l}]\Tr \in T_{\vec{q}_\text{S}}\mathcal{Q}_\text{S}
    \end{array}\right\},\label{eq:X1stance}\\
    &\mathcal{X}_\text{F}=\left\{\vec{x}_\text{F}=\begin{bmatrix}
        y\\ \alpha \\ \dot{x}\\ \dot{y}\\ \dot{\alpha}
    \end{bmatrix}:
    \begin{array}{l}
        [y ~ \alpha]\Tr \in \mathcal{Q}_\text{F},\\\relax
        \dot{x} \in \mathbb{R},\\\relax
        [\dot{y}~\dot{\alpha}]\Tr \in T_{\vec{q}_\text{F}}\mathcal{Q}_\text{F}
    \end{array}\right\},\label{eq:X2flight}\\
    &\vec{f}_\mr{S} =\begin{bmatrix}
        \dot{\alpha}\\ \dot{l}\\ -2\frac{\dot{\alpha}\dot{l}}{l}+\frac{g}{l}\sin(\alpha)\\ l\dot{\alpha}^2 -g\cos(\alpha)-\frac{k}{m_\mr{o}}(l-l_\mr{o})
    \end{bmatrix},\\
    &\vec{f}_\mr{F} =\begin{bmatrix}
        \dot{y} & \dot{\alpha} & 0 & -g & -\alpha \omega^2
    \end{bmatrix}\Tr,\label{eq:fFlight}\\
    &e_\mr{S}^\mr{F} =l_\mr{o}-l,\quad e_\mr{F}^\mr{S} = y-l_\mr{o}\cos(\alpha),\\
&\vec{\Delta}_\mr{S}^\mr{F} = \begin{bmatrix}
         l\cos(\alpha)\\
         \alpha\\
        -\dot{l}\sin(\alpha)-\dot{\alpha}l\cos(\alpha)\\
         \dot{l}\cos(\alpha)-\dot{\alpha}l\sin(\alpha)\\
        \dot{\alpha}
    \end{bmatrix},\\
    &\vec{\Delta}_\mr{F}^\mr{S} = \begin{bmatrix}
        \alpha\\
        l_\mr{o}\\
        -\frac{1}{l_\mr{o}}\left(\cos(\alpha)\dot{x}+\sin(\alpha)\dot{y}\right)\\
        \cos(\alpha)\dot{y}-\sin(\alpha)\dot{x}
    \end{bmatrix},\\
    &H_{\mr{S}} = \tfrac{1}{2}m_\mr{o}(\dot{\alpha}^2 l^2 + \dot{l}^2) +m_\mr{o} g\cos(\alpha)l+\tfrac{1}{2} k(l-l_\mr{o})^2\\
&H_{\mr{F}} = \frac{1}{2}m_\mr{o}(\dot{x}^2+\dot{y}^2)+m_\mr{o}gy,
\end{align}
\end{subequations}}
where the system's energy constitutes a hybrid first integral~$\mathcal{H}=\{H_\text{S},H_\text{F}\}$.
$\Sigma^\mathcal{H}$ starts and ends its flow within the stance phase following the contact sequence $\text{S}\to\text{F}\to\text{S}$.
The configuration space during stance in~\eqref{eq:X1stance} is defined as {$\mathcal{Q}_\text{S}=\mathbb{S}^1\times \mathbb{R}$}, with the leg rotation~$\alpha\in\mathbb{S}^1$ and the leg length~{$l\in\mathbb{R}$}. {Thus,~$\vec{q}_\text{S}\Tr = [\alpha~l]$.}
During flight, the configuration space in~\eqref{eq:X2flight} is defined as {$\mathcal{Q}_\text{F}=\mathbb{R}\times\mathbb{S}^1$}, where {$y\in\mathbb{R}$} denotes the vertical position of the point mass.
{Thus,~$\vec{q}_\text{F}\Tr = [y~\alpha]$.}
{To ensure the physical realizability of the recurrent hybrid trajectory~$\vec{x}(t)$ (\textit{As1}-\textit{As3}), the initial point~$\vec{x}_0 \in \mathcal{X}_0$ must be selected such that~$e_\text{S}^\text{F}(\vec{x}_0) > 0$}.
Note that the two phases have distinct domains, with different sizes:~$n_\text{S}=4$ and~$n_\text{F}=5$.
The horizontal velocity~$\dot{x}$ is constant during flight, as can be seen in the phase dynamics~\eqref{eq:fFlight}.
This implies another independent first integral of phase~F by the linear momentum~$H_{\mr{F},2} = m_\mr{o}\dot{x}$.
However, only~$H_{\mr{F}}$ and~$H_{\mr{S}}$ satisfy Definition~\textit{Df2}.

In this example, the Poincaré section is defined at nadir transit during stance, given the anchor function
\begin{align}
    a(\vec{x}_\text{S})=-\dot{l},\quad \vec{x}_\text{S}=\begin{bmatrix}\alpha & l & \dot{\alpha} & \dot{l}\end{bmatrix}\Tr.
\end{align}
This framework {allows} us to derive an analytical starting point~$\vec{u}_0$ for the time-based implementation~\eqref{eq:timebased}, where {$\vec{u}\Tr=[t_{3}~\bar{\vec{x}}_{3}\Tr~t_{2}~\bar{\vec{x}}_{2}\Tr~t_1~\bar{\vec{x}}_{1}\Tr~\xi~\bar{H}]$}.
For an initial orbit involving purely vertical leg motion ($\alpha\equiv 0$) and a flight time~$t_2=0$, the resulting behavior manifests as a linear oscillation in $l$ with a constant period, denoted as 
\begin{align}
    T_\mr{o}=\sum_{i=1}^3 t_i=2\pi \sqrt{m_\mr{o}/k},
\end{align}
wherein $t_1=t_3=T_\mr{o}/2$.
Furthermore, with 
\begin{align*}
\bar{\vec{x}}_{3}=\begin{bmatrix}
    0\\l_\mr{o}\\0\\0
\end{bmatrix},~\bar{\vec{x}}_{2}=\Delta_\text{S}^\text{F}(\bar{\vec{x}}_{3}),~\bar{\vec{x}}_{1}=\begin{bmatrix}
    0\\l_\mr{o}-\frac{2m_\mr{o}g}{k}\\0\\0
\end{bmatrix},
\end{align*}
the starting point~$\vec{u}_0$ defines an SB point at energy level~$\bar{H}=1~m_\mr{o}gl_\mr{o}$.
In Fig.~\ref{fig:slip}, various projections of the starting point and the bifurcating solutions, obtained via a Lyapunov-Schmidt reduction, are shown. 
The solid blue line represents a normal conservative orbit within~$\Sigma^\mathcal{H}$, whereas the dashed orange curve depicts harmonic oscillations at higher energy levels~$\bar{H}$ with a period of $T_\mr{o} = t_1 + t_3$. {Since the flight duration is zero~($t_2 = 0$), these oscillations are non-physical, violating Assumption~\textit{As3}.}
{Note that the stance durations~$t_1$ and $t_3$ are identical, aside from the harmonic oscillations. In the time-based implementation that generates a harmonic oscillation, the direction of event activation is irrelevant. This allows either $t_1$ or $t_3$ to define the timing of the first or second crossing of~$l = l_\mr{o}$ (Fig.~\ref{fig:slipKeyFrame}b). Consequently, due to the symmetry between~$t_1$ and~$t_3$, they appear identical in the projection shown in Fig.~\ref{fig:slip}b.}
Consistent with previous studies \cite{Gan2018,Raff2022a}, Fig.~\ref{fig:slip}a highlights two additional SB points. 
The introduction of further normal conservative orbits, represented by solid green lines, corresponds to forward and backward hopping motions with~$\dot{x}\neq 0$. 
{Fig.~\ref{fig:slipKeyFrame} illustrates key frames of individual orbits of the hopper, highlighting the main characteristics of the different orbit families shown in Fig.~\ref{fig:slip}.}
Additional SB points emerge along the blue curve for larger period times and for $T<T_\mr{o}$, depicted by the dashed blue line.
The comprehensive identification and presentation of all bifurcation points within this model is beyond the intended scope of this paper.
\begin{remark}
    We purposely neglected the horizontal position~$x$ in the flight dynamics, since it is a cyclic coordinate that can be easily reconstructed by a solution~$\vec{\varphi}(t,\vec{x}_0)$ of the hybrid system~$\Sigma^\mathcal{H}$.
    Unlike the SLIP model without swing dynamics ($\dot{\alpha}=0$) during flight \cite{Blickhan1989,Zaytsev2019}, it is not possible to further reduce~$\Sigma^\mathcal{H}$ to a model with~$m=1$, i.e., a single stance phase as in \cite{Colombo2020}. 
    The reason for this is that there exists not a unique analytic solution for the flight time~$t_\mr{F}^\mathcal{E}$.
    In other words, a function that maps~$\dot{\alpha}_\mr{S}^+\in\mathcal{E}_\mr{S}^\mr{F}$ to~$\dot{\alpha}_\mr{F}^+\in\mathcal{E}_\mr{F}^\mr{S}$ is not unique, since there are known periodic orbits where~$\dot{\alpha}_\mr{S}^+=c$, with~$c\in\mathbb{R}$, yields either~$\dot{\alpha}_\mr{F}^+=c$ or~$\dot{\alpha}_\mr{F}^+=-c$ \cite{Gan2018,Raff2022a}.
\end{remark}

\section{Conclusions}
\label{sec:Discus}
In this work, we have provided a comprehensive framework for understanding conservation in hybrid dynamical systems.
Our contributions build on foundational insights on periodic orbits in conservative ODEs~\cite{Sepulchre1997} and include:
\begin{itemize}
    \item Proving the non-isolation of periodic orbits in cHDS~(Theorem~\ref{thm:continuance}).
    \item Expanding the definition of normal periodic orbits to cHDS.
    \item Embedding a cHDS into a dissipative system while preserving the solution space (Theorem~\ref{theoremXi0}).
\end{itemize}
In addition to these theoretical contributions, a secondary focus of our work has been on the practical application of these results to automatically generate periodic orbits within a numerical continuation framework.
To this end, we extended the state-based formulation \eqref{eq:statebased} to a time-based formulation~\eqref{eq:timebased}.
While the state-based formulation \eqref{eq:statebased} relies on Assumptions \textit{As1}-\textit{As3} and requires defining event manifolds~$\mathcal{E}_i^{i+1}$ and a Poincaré section~$\mathcal{A}$ to ensure the hybrid flow transverses these without grazing, our time-based implementation~\eqref{eq:timebased} offers a more flexible approach.
It only requires the components of a cHDS to be differentiable, thus lifting the need for Assumptions~\textit{As2} and \textit{As3}.
This relaxation proves particularly useful in the numerical continuation process. 
As demonstrated in the examples, it facilitates the detection of bifurcation points and the creation of suitable starting points, as these can be constructed from equilibria with zero phase durations.

The implementation of the time-based methodology necessitates a~$C^1$-differentiable HDS, the identification of a hybrid first integral~$\mathcal{H}$, and an initial periodic orbit~$\vec{u}_0$. 
Deriving an HDS with the necessary differentiability conditions is relatively straightforward in most applications. 
However, identifying~$\mathcal{H}$ remains a general challenge (Remark~\ref{remark:constructIntegral}).
For mechanical systems, we can rely on domain-specific knowledge to determine whether quantities like energy or momentum is preserved, serving as a hybrid first integral~$\mathcal{H}$.
This leaves, as a major restriction of our framework, the construction of an initial periodic orbit~$\vec{u}_0$, which necessitates selecting a fixed repeating phase sequence.
Furthermore, by utilizing numerical continuation methods, we are only able to find periodic orbits that are locally connected to~$\vec{u}_0$. 
Hence, while tracing periodic orbits, we may encounter turning points or asymptotes, such as the homoclinic orbit in the rocking block example, and consequently not covering certain level sets of~$\mathcal{H}$.
The inherently local nature of this approach also limits the exploration of isolated solution families within the solution space~$\mathcal{S}$ that do not connect to~$\vec{u}_0$. Therefore, with numerical continuation methods and an initial periodic orbit~$\vec{u}_0$, a complete description of the solution space~$\mathcal{S}$ of periodic orbits cannot be guaranteed.

When considering potential extensions of our framework, allowing for local changes in the phase sequence could significantly enhance its exploratory capabilities. This would require modifying the map~$\vec{r}$ to encode a different phase sequence whenever phase durations approach zero~($t_i\to 0$), indicating a change in the sequence.
For the continuation process to advance through this map transition, the existence of the same periodic orbit within both maps must be ensured.
Considering, for instance, the hopper example, where altering the phase sequence could enable the connection of the periodic orbit~$\vec{u}_0$, with a zero flight duration~$t_2=0$, to a new map~$\vec{r}$ featuring only two stance phases.
Ensuring the existence of this in-place oscillation as a solution in both maps would further establish a connection between the solution space of the cHDS and periodic orbits of a conservative ODE. 
This connection, achieved through a change in the phase sequence, would significantly enhance the systematic utilization of equilibria for generating initial periodic orbits~$\vec{u}_0$, resembling the concept of nonlinear normal modes~\cite{Kerschen2009,Albu-Schaeffer2020}.

Beyond the potential extensions to the framework, the presented methodology offers a robust tool for systematically exploring families of periodic orbits in cHDS.

\appendix
\section{Saltation Matrix}
\label{sec:appendSaltation}
We will construct the saltation matrix by considering local derivatives of the hybrid flow in phase $i$ and $i+1$.
To this end, we take the total derivative of the phase flow~$\vec{\varphi}_{i+1}$ with respect to $\bar{\vec{x}}_i$ and use the following shorthand notation as before:
\begin{align*}
&\vec{x}_i^-:=\vec{x}_i(t_i^\mathcal{E}(\bar{\vec{x}}_i)),\quad\vec{x}_{i+1}^+:=\bar{\vec{x}}_{i+1}\overset{\eqref{eq:recursion}}{=}\mat{\Delta}_i^{i+1}( \vec{x}_i^-),\\
&\vec{f}_i^{-}:=\vec{f}_i(\vec{x}_i^-),\quad \vec{f}_{i+1}^{+}:=\vec{f}_{i+1}(\vec{x}_{i+1}^+),\\
&\mat{D}_i^{i+1}:=\frac{\partial \mat{\Delta}_i^{i+1}(\vec{x}_i)}{\partial \vec{x}_i}\bigg\vert_{\vec{x}_i^-},~\mr{d} e_{i}^{i+1}:=\frac{\partial e_i^{i+1}(\vec{x}_i)}{\partial \vec{x}_i}\bigg\vert_{\vec{x}_i^-},
\end{align*}
With the time $t\in [t_{i}^\mathcal{E},t_{i}^\mathcal{E}+t_{i+1}^\mathcal{E})$ in phase $i+1$, we also define the local time $\tau:=t-t_{i}^\mathcal{E}(\bar{\vec{x}}_i)$ in the phase flow~$\vec{\varphi}_{i+1}(\tau,\bar{\vec{x}}_{i+1})$, and compute its total derivative
\begin{align}\label{eq:totalDiffPhi2}
&\frac{\mr{d}}{\mr{d}\bar{\vec{x}}_i}\vec{\varphi}_{i+1}(\tau,\bar{\vec{x}}_{i+1})\\
    =~&\frac{\mr{d}}{\mr{d}\bar{\vec{x}}_i}\vec{\varphi}_{i+1}\bigg(t-t_{i}^\mathcal{E}(\bar{\vec{x}}_i),\vec{\Delta}_i^{i+1}\circ \vec{\varphi}_i\left(t_{i}^\mathcal{E}(\bar{\vec{x}}_i),\bar{\vec{x}}_i\right) \bigg)\notag\\
    \overset{\text{c.r.}}{=}~&\frac{\partial\vec{\varphi}_{i+1}}{\partial t}\bigg\vert_{(\tau,\vec{x}_{i+1}^+)}\frac{\partial (t-t_{i}^\mathcal{E})}{\partial t_{i}^\mathcal{E}}\frac{\partial t_{i}^\mathcal{E}(\bar{\vec{x}}_i)}{\partial \bar{\vec{x}}_i} \notag\\
    &+ \frac{\partial\vec{\varphi}_{i+1}}{\partial \vec{x}}\bigg\vert_{(\tau,\vec{x}_{i+1}^+)} \frac{\partial\vec{\Delta}_{i}^{i+1}}{\partial \vec{x}}\bigg\vert_{\vec{x}_{i}^-}\frac{\mr{d}\vec{\varphi}_i\left(t_{i}^\mathcal{E}(\bar{\vec{x}}_i),\bar{\vec{x}}_i\right)}{\mr{d}\bar{\vec{x}}_i}\notag\\
    \overset{\eqref{eq:phaseFprop}}{=}& -\mat{\Phi}_{i+1}(\tau,\vec{x}_{i+1}^+)\vec{f}_{i+1}^+\frac{\partial t_{i}^\mathcal{E}(\bar{\vec{x}}_i)}{\partial \bar{\vec{x}}_i}\notag\\
    ~& +\mat{\Phi}_{i+1}(\tau,\vec{x}_{i+1}^+)\mat{D}_i^{i+1}\frac{\mr{d}}{\mr{d}\bar{\vec{x}}_i}\vec{\varphi}_i\left(t_{i}^\mathcal{E}(\bar{\vec{x}}_i),\bar{\vec{x}}_i\right),\notag
\end{align}
where we applied the chain rule at {c.r.} and used Lemma \ref{lemmaFlow} with the flow property~$\vec{f}_{i+1}(\vec{x}_{i+1}(\tau))=\mat{\Phi}_{i+1}(\tau,\vec{x}_{i+1}^+)\vec{f}_{i+1}(\vec{x}_{i+1}^+)$. The remaining total derivative in \eqref{eq:totalDiffPhi2} takes the form
\begin{align}\label{eq:totalDiffPhi1}
    &\frac{\mr{d}}{\mr{d}\bar{\vec{x}}_i}\vec{\varphi}_i\left(t_{i}^\mathcal{E}(\bar{\vec{x}}_i),\bar{\vec{x}}_i\right)\\
    \overset{\text{c.r.}}{=}~&\frac{\partial\vec{\varphi}_{i}}{\partial t}\bigg\vert_{(t_i^\mathcal{E},\bar{\vec{x}}_i)}\frac{\partial t_{i}^\mathcal{E}(\bar{\vec{x}}_i)}{\partial \bar{\vec{x}}_i}+\frac{\partial\vec{\varphi}_{i}}{\partial \vec{x}}\bigg\vert_{(t_i^\mathcal{E},\bar{\vec{x}}_i)}\notag\\
    =~& \vec{f}_i^- \frac{\partial t_{i}^\mathcal{E}(\bar{\vec{x}}_i)}{\partial \bar{\vec{x}}_i}+\mat{\Phi}_i(t_i^\mathcal{E},\bar{\vec{x}}_i).\notag
\end{align}
Introducing additional shorthand notation
\begin{align*}
&\mat{\Phi}_{i+1}:=\mat{\Phi}_{i+1}(\tau,\vec{x}_{i+1}^+),\quad \mat{\Phi}_i:=\mat{\Phi}_i(t_i^\mathcal{E},\bar{\vec{x}}_i),\\
&\mr{d} t_i^\mathcal{E}:= \frac{\partial t_{i}^\mathcal{E}(\bar{\vec{x}}_i)}{\partial \bar{\vec{x}}_i},
\end{align*}
and plugging \eqref{eq:totalDiffPhi1} into \eqref{eq:totalDiffPhi2}, yields 
\begin{align}\label{eq:totalDiffPhi2plug}
&\frac{\mr{d}}{\mr{d}\bar{\vec{x}}_i}\vec{\varphi}_{i+1}(\tau,\bar{\vec{x}}_{i+1})\\
=~& \mat{\Phi}_{i+1}\bigg(\left(\mat{D}_i^{i+1}\vec{f}_i^- -
\vec{f}_{i+1}^+ \right)\mr{d} t_{i}^\mathcal{E}+\mat{D}_i^{i+1}\mat{\Phi}_i\bigg).\notag
\end{align}
The switching time derivative $\mr{d} t_i^\mathcal{E}$ in \eqref{eq:totalDiffPhi2plug} is well-defined due to Assumption~\textit{As2} and the implicit function theorem.
Hence, the total derivative of the event function $e_i^{i+1}$ yields
\begin{align}
e_i^{i+1}\bigg(\vec{\varphi}_i\left(t_{i}^\mathcal{E}(\bar{\vec{x}}_i),\bar{\vec{x}}_i\right)\bigg)=0\quad \forall \bar{\vec{x}}_i\in\tilde{\mathcal{X}}_i,\notag\\
    \Rightarrow\frac{\mr{d}}{\mr{d}\bar{\vec{x}}_i}e_i^{i+1}\circ\vec{\varphi}_i\left(t_{i}^\mathcal{E}(\bar{\vec{x}}_i),\bar{\vec{x}}_i\right)&=0,\notag\\
    \mr{d} e_i^{i+1}\cdot(\underbrace{\vec{f}_i^- \mr{d} t_i^\mathcal{E}+\mat{\Phi}_i}_{\eqref{eq:totalDiffPhi1}})&=0,\notag\\
\Rightarrow \mr{d} t_i^\mathcal{E}=-\frac{\mr{d} e_i^{i+1}}{\mr{d} e_i^{i+1}\vec{f}_i^-}\mat{\Phi}_i,\label{eq:nablaT}
\end{align}
where $\tilde{\mathcal{X}}_i\subset \mathcal{X}_i$, since $\tilde{\mathcal{X}}_i=\mathcal{X}_0$ when $i=1$ and otherwise $\tilde{\mathcal{X}}_i=\mat{\Delta}_{i-1}^{i}(\mathcal{E}_{i-1}^i)$.
Finally, plugging the switching time derivative \eqref{eq:nablaT} into \eqref{eq:totalDiffPhi2plug}, yields the desired structure
\begin{align*}
&\frac{\mr{d}}{\mr{d}\bar{\vec{x}}_i}\vec{\varphi}_{i+1}(\tau,\bar{\vec{x}}_{i+1})
=\mat{\Phi}_{i+1}\mat{S}_i^{i+1}\mat{\Phi}_i,
\end{align*}
where the saltation matrix $\mat{S}_{i}^{i+1}$ takes the form reported in~\eqref{eq:Saltation}.
\section{Poincaré Section}
\label{sec:appendixPoincare}
In Section~\ref{sec:ConservativeOrbits}, we defined the Poincaré section within the returning phase $m+1=1$.
In certain applications it is common to define the anchor as the last event $a:=e_m^{1}$ \cite{Grizzle2014,Westervelt2018},
since the event manifolds~$\mathcal{E}_i^{i+1}$ share the same transversality to the flow as the Poincaré~section $\mathcal{A}$.
Hence, a hybrid trajectory~$\vec{x}(t^\mathcal{A})$ ends immediately after the discrete transition~$\mat{\Delta}_{m}^{1}$ back to the first phase, with vanishing phase time~$t_{m+1}^\mathcal{E}=0$ such that~$t^\mathcal{A}=t^\mathcal{E}$.
This implies for the initial state~$\vec{x}_0$ to lie within~$\mat{\Delta}_{m}^{1}(\mathcal{E}_{m}^1)\subseteq\mathcal{X}_0$.
It also alters the structure of the Poincaré map's Jacobian defined in \eqref{eq:PoincareJacobian}.
While the fundamental matrix in phase~$m$ is still constructed by the chained structure in~\eqref{eq:fundamentalMatrixHybrid}, the last saltation matrix~$\vec{S}_m^{m+1}$ drops the term~$\vec{f}_{m+1}^+$ in~\eqref{eq:Saltation} such that
\begin{align}\label{eq:LastSaltation}
\tilde{\mat{S}}_{m}^{m+1}=\tilde{\mat{S}}_m^{1}=\mat{D}_{m}^{1}\left(\mat{I} -\frac{\vec{f}_m^- \mr{d} e_{m}^{1}}
    {\mr{d} e_{m}^{1}\vec{f}_m^-}\right).
\end{align} 
Hence, with the choice~$a:=e_m^{1}$, the Poincaré map's Jacobian takes the form
\begin{align}
\begin{aligned}
    \mr{d}\mat{P}(\vec{x}_0) =~&\tilde{\mat{S}}_m^{1}\cdot\mat{\Phi}_{m}(t_m^\mathcal{E}(\bar{\vec{x}}_{m}),\bar{\vec{x}}_{m})\\
    &\cdot\prod_{i=1}^{m-1}\mat{S}_i^{i+1}\cdot\mat{\Phi}_{i}(t_i^\mathcal{E}(\bar{\vec{x}}_i),\bar{\vec{x}}_{i}).
\end{aligned}
\end{align}
The results in Chapter \ref{sec:HDS} and Chapter \ref{sec:Implementation} are not affected by this altered Poincaré map and its Jacobian.
\section{More Proofs}
\label{sec:appendixProof}
In the following, we restate and prove Theorem~\ref{theoremXi0} and Theorem~\ref{theoremRegularPxi}.

\begin{customthm}{3.1}
The tuple $(\vec{x}_0,\xi)$ is a fixed point of $\tilde{\vec{P}}$ if and only if $\xi=0$ and thus, $\vec{x}_0$ is a fixed point of $\vec{P}$. 
\end{customthm}
\begin{proof}
Assume that there exists a periodic motion of $\Sigma_\xi$ such that $\tilde{\vec{P}}(\vec{x}_0,\xi)=\vec{x}_0$ and thus,~$\vec{\varphi}(t^\mathcal{A},\vec{x}_0;\xi)=\vec{x}_0$.
This implies that \begin{align}\label{eq:proof_xi0}
H_1(\vec{\varphi}(t^\mathcal{A},\vec{x}_0;\xi))-H_1(\vec{x}_0)=0.
\end{align}
Furthermore, using the shorthand notation $\vec{x}_i(t):=\vec{\varphi}_i(t,\bar{\vec{x}}_i;\xi)$, for each phase $i$ it holds
\begin{align}
H_i(\vec{x}_i(t_i))-H_i(\bar{\vec{x}}_i)&=\underbrace{\int_0^{t_i}\dfrac{\mr{d}}{\mr{d}t}H_i(\vec{x}_i(t))\mr{d}t}_{=: c_i},
\end{align}
where $t_i=t_{m+1}^\mathcal{A}$ for $i=m+1$ and otherwise $t_i=t_{i}^\mathcal{E}$.
Since all reset maps $\mathcal{D}$ in $\Sigma_\xi$ remain independent of $\xi$, the integral $H_{i+1}(\vec{x}_{i+1}^+) = H_{i}(\vec{x}_{i}^-)$ does not change at phase transitions (\textit{Df2}).
Hence, the left hand side in \eqref{eq:proof_xi0} takes the form
\begin{align}
H_1(\vec{\varphi}(t^\mathcal{A},\vec{x}_0;\xi))-H_1(\vec{x}_0)=\sum_{i=1}^{m+1}c_i\overset{!}{=}0.\label{eq:ci}
\end{align}
Herein, the constant $c_i$ takes the particular form:
\begin{align}
c_i&=\int_0^{t_i}\dfrac{\mr{d}}{\mr{d}t}H_i(\vec{x}_i(t))\mr{d}t=\int_0^{t_i}\mr{d}H_i(\vec{x}_i(t))\tilde{\vec{f}}_i(\vec{x}(t))\mr{d}t\notag\\
&=\int_0^{t_i}\mr{d}H_i(\vec{x}_i(t))\left(\vec{f}_i(\vec{x}_i(t))+ \xi\cdot \mr{d}H_i(\vec{x}_i(t))\Tr\right)\mr{d}t\notag\\
&=\xi\cdot\underbrace{\int_0^{t_i}\left\Vert \mr{d}H_i(\vec{x}_i(t))\right\Vert_2^2\mr{d}t}_{=:\tilde{c}_i}\label{eq:ciTilde},
\end{align}
where we used~$\mr{d}H_i(\vec{x}_i)\vec{f}_i(\vec{x}_i)=0$, since the property in \eqref{eq:orthogonalFirstIntegral} also holds for any~$\bar{\vec{x}}_i\in\mathcal{X}_i$ at initial time.
Since~$\mr{d}H_i(\vec{x}_i)\neq \vec{0}$~(Lemma~\ref{lemma:nonzerodh}) and~$t_i>0$ (\textit{As3}) it follows that~$\tilde{c}_i>0$.
Finally, plugging \eqref{eq:ciTilde} into the condition \eqref{eq:ci} yields~$\xi=0$ for a periodic motion of~$\Sigma_\xi$.
\end{proof}

\begin{customthm}{3.2}
Given a normal conservative orbit such that~$\vec{x}_0$ is a fixed point of the Poincaré map~$\vec{P}$, $\vec{u}\Tr=\left[\vec{x}_0\Tr~0~H_1(\vec{x}_0)\right]$ is a regular point of~$\vec{r}$ in \eqref{eq:statebased}.
\end{customthm}
\begin{proof}
For regularity, we need to show that the Jacobian 
\begin{align}\label{eq:druAppendix}
    \mr{d}\vec{r}(\vec{u})=\begin{bmatrix}
    \dfrac{\partial \tilde{\vec{P}}}{\partial \vec{x}_0}\bigg\vert_{(\vec x_0,0)}-\vec{I} & \quad\dfrac{\partial \tilde{\vec{P}}}{\partial \xi}\bigg\vert_{(\vec x_0,0)} & 0\\ \mr{d}H_1(\vec{x}_0) & 0 & -1\end{bmatrix},
\end{align}
has full rank.
This is equivalent to showing that the tangent vector~$\vec{\tau}\circ\mr{d}\vec{r}(\vec{u})$ is unique or~$\text{dim}\left(\text{ker}\left(\mr{d}\vec{r}(\vec{u})\right)\right)=1$.
We consider the tangent vector to be of the form
\begin{align}
\vec{\tau}=\begin{bmatrix}
    \vec{\tau}_{x_0}\\\tau_{\xi}\\\tau_{\bar{H}}
\end{bmatrix},\quad \vec{\tau}_{x_0}\in\mathbb{R}^{n_1},~\tau_{\xi}\in\mathbb{R},~\tau_{\bar{H}}\in\mathbb{R}.
\end{align}
Since~$\xi$ remains zero for normal conservative orbits of~$\vec{r}^{-1}(\vec{0})$ (Theorem~\ref{theoremXi0}), it holds~$\tau_\xi=0$.
The other vector components~$\vec{\tau}_{x_0}$ and~$\tau_{\bar{H}}$ are coupled by the last row of~$\mr{d}\vec{r}(\vec{u})$, i.e., $\mr{d}\bar{H}_1(\vec{x}_0)\vec{\tau}_{x_0}=\tau_{\bar{H}}$.
Hence, $\vec{\tau}_{x_0}=\vec{0}$ implies~$\tau_{\bar{H}}$ and thus, the trivial tangent vector~$\vec{\tau}=\vec{0}$.
Consequently, we solely consider nonzero~$\vec{\tau}_{x_0}$ in~$\mr{d}\vec{r}(\vec{u})\vec{\tau}=\vec{0}$.
Given that~$\vec{x}_0$ lies on a normal conservative orbit, we know that~$\text{dim}(\text{ker}(\mr{d}\mat{P}(\vec{x}_0)-\vec{I}))=1$.
And since~$\xi=0$, the Jacobian~$\mr{d}\mat{P}(\vec{x}_0)$ is equivalent to~$\nicefrac{\partial \tilde{\vec{P}}}{\partial \vec{x}_0}\vert_{(\vec x_0,0)}$ in \eqref{eq:druAppendix}.
This implies that there only exists a single linearly independent vector~$\vec{\tau}_{x_0}\neq \vec{0}$ in the kernel of~$\nicefrac{\partial \tilde{\vec{P}}}{\partial \vec{x}_0}\vert_{(\vec x_0,0)}$.
Since~$\vec{\tau}$ is the only linearly independent vector in the kernel of~$\mr{d}\vec{r}(\vec{u})$, we conclude~$\text{dim}(\text{ker}(\mr{d}\vec{r}(\vec{u}))=1$.
\end{proof}

\printbibliography

\end{document}

%% file: styleFile.tex
\usepackage{lmodern}
\usepackage[T1]{fontenc}
\usepackage[natbib=true, style=numeric, sorting=none, giveninits=true, maxnames=6, useprefix=true]{biblatex}
\DeclareNameAlias{default}{family-given/given-family}

\DeclareFieldFormat[article]{volume}{\mkbibbold{#1}}
\renewbibmacro*{volume+number+eid}{%
  \printfield{volume}%
  \iffieldundef{number}
    {}
    {(\printfield{number})}%
  \setunit{\addcomma\space}%
  \printfield{eid}}
\DeclareSourcemap{
  \maps[datatype=bibtex]{
    \map{
      \step[fieldset=month, null]
    }
  }
}
\AtEveryBibitem{\clearfield{issn}}
\AtEveryBibitem{\clearfield{isbn}}

\usepackage[
    paperwidth=193mm,
    paperheight=260mm,
    top=19.29mm,
    headheight=12pt,
    headsep=5.15mm,
    text={160mm,216mm},
    marginparsep=5mm,
    marginparwidth=12mm,
    footskip=10.13mm,
    twocolumn,
    columnsep=8mm
]{geometry}

\renewcommand\footnoterule{%
  \kern3pt%
  \hrule width \columnwidth height 0.2mm
  \kern5.5pt%
}
\providecommand{\keywords}[1]
{
  \small	
  \\[2mm] \textbf{Keywords:} #1 \\[2mm]
}

\newcommand*\samethanks[1][\value{footnote}]{\footnotemark[#1]}

\AtBeginEnvironment{proof}{\small}

\usepackage{titlesec}
\titleformat{\section}
  {\normalfont\bfseries\fontsize{14pt}{16pt}\selectfont\raggedright}
  {\thesection}{1em}{}
  
\titleformat{\subsection}
  {\normalfont\bfseries\fontsize{12pt}{14pt}\selectfont\raggedright}
  {\thesubsection}{1em}{}

\titleformat{\subsubsection}
  {\normalfont\bfseries\fontsize{11pt}{13pt}\selectfont\raggedright}
  {\thesubsubsection}{1em}{}

%% file: literature.bib
@Article{Raff2022a,
  author  = {Raff, Maximilian and Rosa, Nelson and Remy, C. David},
  title   = {Connecting gaits in energetically conservative legged systems},
  journal = {IEEE Robotics and Automation Letters},
  year    = {2022},
  volume  = {7},
  number  = {3},
  pages   = {8407-8414},
  doi     = {10.1109/LRA.2022.3186500},
}

@Book{Arnold1992,
  title     = {Ordinary differential equations},
  publisher = {Springer Berlin, Heidelberg},
  year      = {1992},
  author    = {Vladimir I. Arnold},
  month     = {may},
}

@Article{Bizzarri2016,
  author    = {Federico Bizzarri and Alessandro Colombo and Fabio Dercole and Giancarlo Storti Gajani},
  title     = {Necessary and sufficient conditions for the noninvertibility of fundamental solution matrices of a discontinuous system},
  journal   = {{SIAM} Journal on Applied Dynamical Systems},
  year      = {2016},
  volume    = {15},
  number    = {1},
  pages     = {84--105},
  month     = {jan},
  doi       = {10.1137/140959031},
  publisher = {Society for Industrial {\&} Applied Mathematics ({SIAM})},
}

@Article{Colombo2020,
  author    = {Leonardo J. Colombo and Mar{\'{\i}}a Emma Eyrea Iraz{\'{u}}},
  title     = {Symmetries and periodic orbits in simple hybrid Routhian systems},
  journal   = {Nonlinear Analysis: Hybrid Systems},
  year      = {2020},
  volume    = {36},
  pages     = {100857},
  month     = {may},
  doi       = {10.1016/j.nahs.2020.100857},
  publisher = {Elsevier {BV}},
}

@PhdThesis{Ames2006,
  author = {Ames, Aaron David},
  title  = {A categorical theory of hybrid systems},
  school = {Citeseer},
  year   = {2006},
  type   = {PhD thesis},
  url    = {http://ames.caltech.edu/A%20categorical%20theory.pdf},
}

@Article{Blickhan1989,
  author    = {R. Blickhan},
  title     = {The spring-mass model for running and hopping},
  journal   = {Journal of Biomechanics},
  year      = {1989},
  volume    = {22},
  number    = {11-12},
  pages     = {1217--1227},
  month     = {jan},
  doi       = {10.1016/0021-9290(89)90224-8},
  publisher = {Elsevier {BV}},
}

@Article{Zaytsev2019,
  author    = {Petr Zaytsev and Tom Cnops and C. David Remy},
  title     = {A detailed look at the {SLIP} model dynamics: Bifurcations, chaotic behavior, and fractal basins of attraction},
  journal   = {Journal of Computational and Nonlinear Dynamics},
  year      = {2019},
  volume    = {14},
  number    = {8},
  month     = {may},
  doi       = {10.1115/1.4043453},
  publisher = {{ASME} International},
}

@Article{Lamb1998,
  author    = {Jeroen S.W. Lamb and John A.G. Roberts},
  title     = {Time-reversal symmetry in dynamical systems: A survey},
  journal   = {Physica D: Nonlinear Phenomena},
  year      = {1998},
  volume    = {112},
  number    = {1-2},
  pages     = {1--39},
  month     = {jan},
  doi       = {10.1016/s0167-2789(97)00199-1},
  publisher = {Elsevier {BV}},
}

@Article{Naz2014,
  author    = {Rehana Naz and Igor Leite Freire and Imran Naeem},
  title     = {Comparison of different approaches to construct first integrals for ordinary differential equations},
  journal   = {Abstract and Applied Analysis},
  year      = {2014},
  volume    = {2014},
  pages     = {1--15},
  doi       = {10.1155/2014/978636},
  publisher = {Hindawi Limited},
}

@Book{Krauskopf2007,
  title     = {Numerical continuation methods for dynamical systems},
  publisher = {Springer Netherlands},
  year      = {2007},
  editor    = {Bernd Krauskopf and Hinke M. Osinga and Jorge Gal{\'{a}}n-Vioque},
  doi       = {10.1007/978-1-4020-6356-5},
}

@Article{Gorni2022,
  author    = {Gianluca Gorni and Mattia Scomparin and Gaetano Zampieri},
  title     = {Nonlocal constants of motion in Lagrangian dynamics of any order},
  journal   = {Partial Differential Equations in Applied Mathematics},
  year      = {2022},
  volume    = {5},
  pages     = {100262},
  month     = {jun},
  doi       = {10.1016/j.padiff.2022.100262},
  publisher = {Elsevier {BV}},
}

@Article{Gorni2021,
  author    = {Gianluca Gorni and Gaetano Zampieri},
  title     = {Revisiting Noether{\textquotesingle}s theorem on constants of motion},
  journal   = {Journal of Nonlinear Mathematical Physics},
  year      = {2021},
  volume    = {21},
  number    = {1},
  pages     = {43},
  doi       = {10.1080/14029251.2014.894720},
  publisher = {Springer Science and Business Media {LLC}},
}

@Article{Albu-Schaeffer2020,
  author    = {Alin Albu-Schäffer and Cosimo Della Santina},
  title     = {A review on nonlinear modes in conservative mechanical systems},
  journal   = {Annual Reviews in Control},
  year      = {2020},
  volume    = {50},
  pages     = {49--71},
  doi       = {10.1016/j.arcontrol.2020.10.002},
  publisher = {Elsevier {BV}},
}

@Article{Kerschen2009,
  author    = {G. Kerschen and M. Peeters and J.C. Golinval and A.F. Vakakis},
  title     = {Nonlinear normal modes, Part I: A useful framework for the structural dynamicist},
  journal   = {Mechanical Systems and Signal Processing},
  year      = {2009},
  volume    = {23},
  number    = {1},
  pages     = {170--194},
  month     = {jan},
  doi       = {10.1016/j.ymssp.2008.04.002},
  publisher = {Elsevier {BV}},
}

@Article{Kong2024,
  author    = {Kong, Nathan J. and Payne, J. Joe and Zhu, James and Johnson, Aaron M.},
  title     = {Saltation matrices: The essential tool for linearizing hybrid dynamical systems},
  journal   = {Proceedings of the IEEE},
  year      = {2024},
  pages     = {1--24},
  issn      = {1558-2256},
  doi       = {10.1109/jproc.2024.3440211},
  publisher = {Institute of Electrical and Electronics Engineers (IEEE)},
}

@Article{Burden2016,
  author    = {Samuel A. Burden and S. Shankar Sastry and Daniel E. Koditschek and Shai Revzen},
  title     = {Event--selected vector field discontinuities yield piecewise--differentiable flows},
  journal   = {{SIAM} Journal on Applied Dynamical Systems},
  year      = {2016},
  volume    = {15},
  number    = {2},
  pages     = {1227--1267},
  month     = {jan},
  doi       = {10.1137/15m1016588},
  publisher = {Society for Industrial {\&} Applied Mathematics ({SIAM})},
}

@Article{Munoz-Almaraz2003,
  author    = {F.J. Mu{\~{n}}oz-Almaraz and E. Freire and J. Gal{\'{a}}n and E. Doedel and A. Vanderbauwhede},
  title     = {Continuation of periodic orbits in conservative and Hamiltonian systems},
  journal   = {Physica D: Nonlinear Phenomena},
  year      = {2003},
  volume    = {181},
  number    = {1-2},
  pages     = {1--38},
  month     = {jul},
  doi       = {10.1016/s0167-2789(03)00097-6},
  publisher = {Elsevier {BV}},
}

@Article{Sepulchre1997,
  author    = {Jacques-Alexandre Sepulchre and Robert S MacKay},
  title     = {Localized oscillations in conservative or dissipative networks of weakly coupled autonomous oscillators},
  journal   = {Nonlinearity},
  year      = {1997},
  volume    = {10},
  number    = {3},
  pages     = {679--713},
  month     = {may},
  doi       = {10.1088/0951-7715/10/3/006},
  publisher = {{IOP} Publishing},
}

@Book{Allgower2003,
  title     = {Introduction to numerical continuation methods},
  publisher = {Society for Industrial and Applied Mathematics},
  year      = {2003},
  author    = {Eugene L. Allgower and Kurt Georg},
  month     = {jan},
  doi       = {10.1137/1.9780898719154},
}

@Book{Hartman2002,
  title     = {Ordinary differential equations},
  publisher = {Society for Industrial and Applied Mathematics},
  year      = {2002},
  author    = {Philip Hartman},
  month     = {jan},
  doi       = {10.1137/1.9780898719222},
}

@Article{Grizzle2014,
  author    = {Jessy W. Grizzle and Christine Chevallereau and Ryan W. Sinnet and Aaron D. Ames},
  title     = {Models, feedback control, and open problems of 3D bipedal robotic walking},
  journal   = {Automatica},
  year      = {2014},
  volume    = {50},
  number    = {8},
  pages     = {1955--1988},
  month     = {aug},
  doi       = {10.1016/j.automatica.2014.04.021},
  publisher = {Elsevier {BV}},
}

@Book{Leine2004,
  title     = {Dynamics and bifurcations of non-smooth mechanical systems},
  publisher = {Springer Berlin Heidelberg},
  year      = {2004},
  author    = {Remco I. Leine and Henk Nijmeijer},
  doi       = {10.1007/978-3-540-44398-8},
}

@Article{Henderson2002,
  author    = {Henderson, Michael E},
  title     = {Multiple parameter continuation: Computing implicitly defined k-manifolds},
  journal   = {International Journal of Bifurcation and Chaos},
  year      = {2002},
  volume    = {12},
  number    = {03},
  pages     = {451--476},
  month     = {mar},
  doi       = {10.1142/s0218127402004498},
  publisher = {World Scientific Pub Co Pte Lt},
}

@Article{Gan2018,
  author    = {Zhenyu Gan and Yevgeniy Yesilevskiy and Petr Zaytsev and C. David Remy},
  title     = {All common bipedal gaits emerge from a single passive model},
  journal   = {Journal of The Royal Society Interface},
  year      = {2018},
  volume    = {15},
  number    = {146},
  pages     = {20180455},
  month     = {sep},
  doi       = {10.1098/rsif.2018.0455},
  publisher = {The Royal Society},
}

@Article{Kozlov2020,
  author    = {V. V. Kozlov},
  title     = {First integrals and asymptotic trajectories},
  journal   = {Sbornik: Mathematics},
  year      = {2020},
  volume    = {211},
  number    = {1},
  pages     = {29--54},
  month     = {jan},
  doi       = {10.1070/sm9291},
  publisher = {Steklov Mathematical Institute},
}

@InProceedings{Ames2005,
  author    = {A.D. Ames and A. Abate and S. Sastry},
  title     = {Sufficient conditions for the existence of Zeno behavior},
  booktitle = {Proceedings of the 44th {IEEE} Conference on Decision and Control},
  year      = {2005},
  publisher = {{IEEE}},
  doi       = {10.1109/cdc.2005.1582237},
}

@Article{Burden2015,
  author    = {Samuel A. Burden and Shai Revzen and S. Shankar Sastry},
  title     = {Model reduction near periodic orbits of hybrid dynamical systems},
  journal   = {{IEEE} Transactions on Automatic Control},
  year      = {2015},
  volume    = {60},
  number    = {10},
  pages     = {2626--2639},
  month     = {oct},
  doi       = {10.1109/tac.2015.2411971},
  publisher = {Institute of Electrical and Electronics Engineers ({IEEE})},
}

@InProceedings{Pace2017,
  author    = {Andrew M. Pace and Samuel A. Burden},
  title     = {Piecewise - differentiable trajectory outcomes in mechanical systems subject to unilateral constraints},
  booktitle = {Proceedings of the 20th International Conference on Hybrid Systems: Computation and Control},
  year      = {2017},
  month     = {apr},
  publisher = {{ACM}},
  doi       = {10.1145/3049797.3049807},
}

@Article{Dieci2011,
  author    = {Luca Dieci and Luciano Lopez},
  title     = {Fundamental matrix solutions of piecewise smooth differential systems},
  journal   = {Mathematics and Computers in Simulation},
  year      = {2011},
  volume    = {81},
  number    = {5},
  pages     = {932--953},
  month     = {jan},
  doi       = {10.1016/j.matcom.2010.10.012},
  publisher = {Elsevier {BV}},
}

@InProceedings{Hyon2005,
  author    = {S. Hyon and T. Emura},
  title     = {Symmetric walking control: invariance and global stability},
  booktitle = {Proceedings of the 2005 {IEEE} International Conference on Robotics and Automation},
  year      = {2005},
  publisher = {{IEEE}},
  doi       = {10.1109/robot.2005.1570318},
}

@InProceedings{Ames2006a,
  author    = {A.D. Ames and S. Sastry},
  title     = {Hybrid cotangent bundle reduction of simple hybrid mechanical systems with symmetry},
  booktitle = {2006 American Control Conference},
  year      = {2006},
  publisher = {{IEEE}},
  doi       = {10.1109/acc.2006.1656622},
}

@InProceedings{Ames2006b,
  author    = {A.D. Ames and S. Sastry},
  title     = {Hybrid Routhian reduction of Lagrangian hybrid systems},
  booktitle = {2006 American Control Conference},
  year      = {2006},
  publisher = {{IEEE}},
  doi       = {10.1109/acc.2006.1656621},
}

@Article{Mueller1995,
  author    = {Peter C. M{\"u}ller},
  title     = {Calculation of lyapunov exponents for dynamic systems with discontinuities},
  journal   = {Chaos, Solitons \& Fractals},
  year      = {1995},
  volume    = {5},
  number    = {9},
  pages     = {1671--1681},
  month     = {sep},
  doi       = {10.1016/0960-0779(94)00170-u},
  publisher = {Elsevier {BV}},
}

@Article{Goebel2009,
  author    = {Rafal Goebel and Ricardo G. Sanfelice and Andrew R. Teel},
  title     = {Hybrid dynamical systems},
  journal   = {{IEEE} Control Systems},
  year      = {2009},
  volume    = {29},
  number    = {2},
  pages     = {28--93},
  month     = {apr},
  doi       = {10.1109/mcs.2008.931718},
  publisher = {Institute of Electrical and Electronics Engineers ({IEEE})},
}

@Article{Johnson2016,
  author    = {Aaron M Johnson and Samuel A Burden and Daniel E Koditschek},
  title     = {A hybrid systems model for simple manipulation and self-manipulation systems},
  journal   = {The International Journal of Robotics Research},
  year      = {2016},
  volume    = {35},
  number    = {11},
  pages     = {1354--1392},
  month     = {may},
  doi       = {10.1177/0278364916639380},
  publisher = {{SAGE} Publications},
}

@Book{Westervelt2018,
  title     = {Feedback control of dynamic bipedal robot locomotion},
  publisher = {{CRC} Press},
  year      = {2018},
  author    = {Eric R. Westervelt and Jessy W. Grizzle and Christine Chevallereau and Jun Ho Choi and Benjamin Morris},
  month     = {oct},
  doi       = {10.1201/9781420053739},
}

@Article{Westervelt2003,
  author    = {E.R. Westervelt and J.W. Grizzle and D.E. Koditschek},
  title     = {Hybrid zero dynamics of planar biped walkers},
  journal   = {{IEEE} Transactions on Automatic Control},
  year      = {2003},
  volume    = {48},
  number    = {1},
  pages     = {42--56},
  month     = {jan},
  doi       = {10.1109/tac.2002.806653},
  publisher = {Institute of Electrical and Electronics Engineers ({IEEE})},
}

@Article{Wendel2012,
  author    = {Eric Wendel and Aaron D. Ames},
  title     = {Rank deficiency and superstability of hybrid systems},
  journal   = {Nonlinear Analysis: Hybrid Systems},
  year      = {2012},
  volume    = {6},
  number    = {2},
  pages     = {787--805},
  month     = {may},
  doi       = {10.1016/j.nahs.2011.09.002},
  publisher = {Elsevier {BV}},
}

@Article{Doedel2003,
  author    = {E. J. Doedel and R. C. Paffenroth and H. B. Keller and D. J. Dichmann and J. Gal{\'{a}}n-Vioque and A. Vanderbauwhede},
  title     = {Computation of periodic solutions of conservative systems with application to the 3-body problem},
  journal   = {International Journal of Bifurcation and Chaos},
  year      = {2003},
  volume    = {13},
  number    = {06},
  pages     = {1353--1381},
  month     = {jun},
  doi       = {10.1142/s0218127403007291},
  publisher = {World Scientific Pub Co Pte Lt},
}

@Article{Shampine2000,
  author    = {L.F. Shampine and S. Thompson},
  title     = {Event location for ordinary differential equations},
  journal   = {Computers \& Mathematics with Applications},
  year      = {2000},
  volume    = {39},
  number    = {5-6},
  pages     = {43--54},
  month     = {mar},
  doi       = {10.1016/s0898-1221(00)00045-6},
  publisher = {Elsevier {BV}},
}

@InProceedings{Rosa2014,
  author    = {Nelson Rosa and Kevin M. Lynch},
  title     = {Extending equilibria to periodic orbits for walkers using continuation methods},
  booktitle = {2014 {IEEE}/{RSJ} International Conference on Intelligent Robots and Systems},
  year      = {2014},
  month     = {sep},
  publisher = {{IEEE}},
  doi       = {10.1109/iros.2014.6943076},
}

@Article{Dickinson1976,
  author    = {Robert P. Dickinson and Robert J. Gelinas},
  title     = {Sensitivity analysis of ordinary differential equation systems{\textemdash}A direct method},
  journal   = {Journal of Computational Physics},
  year      = {1976},
  volume    = {21},
  number    = {2},
  pages     = {123--143},
  month     = {jun},
  doi       = {10.1016/0021-9991(76)90007-3},
  publisher = {Elsevier {BV}},
}

@Book{Ortega2000,
  title     = {Iterative solution of nonlinear equations in several variables},
  publisher = {Society for Industrial and Applied Mathematics},
  year      = {2000},
  author    = {J. M. Ortega and W. C. Rheinboldt},
  month     = {jan},
  doi       = {10.1137/1.9780898719468},
}

@Article{Charalampakis2022,
  author    = {Aristotelis E. Charalampakis and George C. Tsiatas and Panos Tsopelas},
  title     = {New insights on rocking of rigid blocks: Analytical solutions and exact energy-based overturning criteria},
  journal   = {Earthquake Engineering \& Structural Dynamics},
  year      = {2022},
  volume    = {51},
  number    = {9},
  pages     = {1965--1993},
  month     = {apr},
  doi       = {10.1002/eqe.3649},
  publisher = {Wiley},
}

@Article{Pontryagin1962,
  author    = {L. S. Pontryagin},
  title     = {Ordinary differential equations},
  journal   = {Addison-Wesley},
  year      = {1962},
  publisher = {Wiley},
}

@Book{Christopher2007,
  title     = {Limit cycles of differential equations},
  publisher = {Birkhäuser Basel},
  year      = {2007},
  author    = {Colin Christopher, Chengzhi Li},
  edition   = {1},
  isbn      = {9783764384098},
  doi       = {10.1007/978-3-7643-8410-4},
  journal   = {Advanced Courses in Mathematics CRM Barcelona},
}

@Article{Chatterjee2002,
  author    = {Chatterjee, A. and Pratap, R. and Reddy, C.K. and Ruina, A.},
  title     = {Persistent passive hopping and juggling is possible even with plastic collisions},
  journal   = {The International Journal of Robotics Research},
  year      = {2002},
  volume    = {21},
  number    = {7},
  pages     = {621--634},
  month     = jul,
  issn      = {1741-3176},
  doi       = {10.1177/027836402322023213},
  publisher = {SAGE Publications},
}

@Book{Gomes2005,
  title     = {Collisionless rigid body locomotion models and physically based homotopy methods for finding periodic motions in high degree of freedom models},
  publisher = {Cornell University},
  year      = {2005},
  author    = {Gomes, Mario Waldorff},
}

@Misc{Keller1987,
  author  = {Keller, Herbert B},
  title   = {Lectures on numerical methods in bifurcation problems},
  year    = {1987},
  journal = {Tata Institute Of Fundamental Research Springer-Verlag Berlin Heidelberg New York Tokyo},
  pages   = {140},
  url     = {https://www.math.tifr.res.in/~publ/ln/tifr79.pdf},
}

@Book{Schaft2000,
  title     = {An introduction to hybrid dynamical systems},
  publisher = {Springer London},
  year      = {2000},
  author    = {van der Schaft, Arjan and Schumacher, Hans},
  isbn      = {9781846285424},
  doi       = {10.1007/bfb0109998},
  issn      = {1610-7411},
  journal   = {Lecture Notes in Control and Information Sciences},
}

@Article{Aihara2010,
  author    = {Aihara, Kazuyuki and Suzuki, Hideyuki},
  title     = {Theory of hybrid dynamical systems and its applications to biological and medical systems},
  journal   = {Philosophical Transactions of the Royal Society A: Mathematical, Physical and Engineering Sciences},
  year      = {2010},
  volume    = {368},
  number    = {1930},
  pages     = {4893--4914},
  month     = nov,
  issn      = {1471-2962},
  doi       = {10.1098/rsta.2010.0237},
  publisher = {The Royal Society},
}

@Article{Lygeros2003,
  author    = {Lygeros, J. and Johansson, K.H. and Simic, S.N. and Jun Zhang and Sastry, S.S.},
  title     = {Dynamical properties of hybrid automata},
  journal   = {IEEE Transactions on Automatic Control},
  year      = {2003},
  volume    = {48},
  number    = {1},
  pages     = {2--17},
  month     = jan,
  issn      = {0018-9286},
  doi       = {10.1109/tac.2002.806650},
  publisher = {Institute of Electrical and Electronics Engineers (IEEE)},
}

@Article{Tomlin2000,
  author    = {Tomlin, C.J. and Lygeros, J. and Shankar Sastry, S.},
  title     = {A game theoretic approach to controller design for hybrid systems},
  journal   = {Proceedings of the IEEE},
  year      = {2000},
  volume    = {88},
  number    = {7},
  pages     = {949--970},
  month     = jul,
  issn      = {1558-2256},
  doi       = {10.1109/5.871303},
  publisher = {Institute of Electrical and Electronics Engineers (IEEE)},
}

@Article{Simic2002,
  author    = {Simic, Slobodan N. and Sastry, Shankar and Johansson, Karl Henrik and Lygeros, John},
  title     = {Hybrid limit cycles and hybrid Poincaré-Bendixson},
  journal   = {IFAC Proceedings Volumes},
  year      = {2002},
  volume    = {35},
  number    = {1},
  pages     = {197--202},
  issn      = {1474-6670},
  doi       = {10.3182/20020721-6-es-1901.01104},
  publisher = {Elsevier BV},
}

@Article{Housner1963,
  author    = {Housner, George W.},
  title     = {The behavior of inverted pendulum structures during earthquakes},
  journal   = {Bulletin of the Seismological Society of America},
  year      = {1963},
  volume    = {53},
  number    = {2},
  pages     = {403--417},
  month     = feb,
  issn      = {0037-1106},
  doi       = {10.1785/bssa0530020403},
  publisher = {Seismological Society of America (SSA)},
}

@Article{Geyer2006,
  author    = {Geyer, Hartmut and Seyfarth, Andre and Blickhan, Reinhard},
  title     = {Compliant leg behaviour explains basic dynamics of walking and running},
  journal   = {Proceedings of the Royal Society B: Biological Sciences},
  year      = {2006},
  volume    = {273},
  number    = {1603},
  pages     = {2861--2867},
  month     = aug,
  issn      = {1471-2954},
  doi       = {10.1098/rspb.2006.3637},
  publisher = {The Royal Society},
}

@Article{Garcia1998,
  author    = {Garcia, Mariano and Chatterjee, Anindya and Ruina, Andy and Coleman, Michael},
  title     = {The simplest walking model: stability, complexity, and scaling},
  journal   = {Journal of Biomechanical Engineering},
  year      = {1998},
  volume    = {120},
  number    = {2},
  pages     = {281--288},
  month     = apr,
  issn      = {1528-8951},
  doi       = {10.1115/1.2798313},
  publisher = {ASME International},
}

@Article{Kuo2001,
  author    = {Kuo, Arthur D.},
  title     = {A simple model of bipedal walking predicts the preferred speed–step length relationship},
  journal   = {Journal of Biomechanical Engineering},
  year      = {2001},
  volume    = {123},
  number    = {3},
  pages     = {264--269},
  month     = jan,
  issn      = {1528-8951},
  doi       = {10.1115/1.1372322},
  publisher = {ASME International},
}

@Article{Blickhan1993,
  author    = {Blickhan, R. and Full, R.J.},
  title     = {Similarity in multilegged locomotion: Bouncing like a monopode},
  journal   = {Journal of Comparative Physiology A},
  year      = {1993},
  volume    = {173},
  number    = {5},
  month     = nov,
  issn      = {1432-1351},
  doi       = {10.1007/bf00197760},
  publisher = {Springer Science and Business Media LLC},
}

@Article{McMahon1987,
  author    = {McMahon, T. A. and Valiant, G. and Frederick, E. C.},
  title     = {Groucho running},
  journal   = {Journal of Applied Physiology},
  year      = {1987},
  volume    = {62},
  number    = {6},
  pages     = {2326--2337},
  month     = jun,
  issn      = {1522-1601},
  doi       = {10.1152/jappl.1987.62.6.2326},
  publisher = {American Physiological Society},
}

@Article{Remy2009,
  author    = {Remy, C. David and Buffinton, Keith and Siegwart, Roland},
  title     = {Stability analysis of passive dynamic walking of quadrupeds},
  journal   = {The International Journal of Robotics Research},
  year      = {2009},
  volume    = {29},
  number    = {9},
  pages     = {1173--1185},
  month     = aug,
  issn      = {1741-3176},
  doi       = {10.1177/0278364909344635},
  publisher = {SAGE Publications},
}

@Article{McGeer1990,
  author    = {McGeer, Tad},
  title     = {Passive dynamic walking},
  journal   = {The International Journal of Robotics Research},
  year      = {1990},
  volume    = {9},
  number    = {2},
  pages     = {62--82},
  month     = apr,
  issn      = {1741-3176},
  doi       = {10.1177/027836499000900206},
  publisher = {SAGE Publications},
}

@PhdThesis{Barrow-Green1997,
  author    = {Barrow-Green, June Elizabeth},
  title     = {Poincaré and the three body problem},
  year      = {1997},
  type      = {PhD thesis},
  doi       = {10.21954/OU.RO.0000E03B},
  number    = {11},
  publisher = {American Mathematical Soc.},
}

@Book{Guckenheimer1983,
  title     = {Nonlinear oscillations, dynamical systems, and bifurcations of vector fields},
  publisher = {Springer New York},
  year      = {1983},
  author    = {Guckenheimer, John and Holmes, Philip},
  isbn      = {9781461211402},
  doi       = {10.1007/978-1-4612-1140-2},
  issn      = {2196-968X},
  journal   = {Applied Mathematical Sciences},
}

@Article{Montenbruck2002,
  author    = {Montenbruck, O and Gill, E and Lutze, FH},
  title     = {Satellite orbits: models, methods, and applications},
  journal   = {Applied Mechanics Reviews},
  year      = {2002},
  volume    = {55},
  number    = {2},
  pages     = {B27–B28},
  month     = mar,
  issn      = {2379-0407},
  doi       = {10.1115/1.1451162},
  publisher = {ASME International},
}

@Article{Razavi2016,
  author    = {Razavi, Hamed and Bloch, Anthony M. and Chevallereau, Christine and Grizzle, Jessy W.},
  title     = {Symmetry in legged locomotion: A new method for designing stable periodic gaits},
  journal   = {Autonomous Robots},
  year      = {2016},
  volume    = {41},
  number    = {5},
  pages     = {1119--1142},
  month     = {July},
  issn      = {1573-7527},
  doi       = {10.1007/s10514-016-9593-x},
  publisher = {Springer Science and Business Media LLC},
}

@Article{Rosa2022,
  author    = {Rosa, Nelson and Lynch, Kevin M.},
  title     = {A topological approach to gait generation for biped robots},
  journal   = {IEEE Transactions on Robotics},
  year      = {2022},
  volume    = {38},
  number    = {2},
  pages     = {699--718},
  month     = apr,
  issn      = {1941-0468},
  doi       = {10.1109/tro.2021.3094159},
  publisher = {Institute of Electrical and Electronics Engineers (IEEE)},
}
